\documentclass[preprint,12pt]{elsarticle}

\usepackage{amssymb}
\usepackage{graphicx, amsmath}
\usepackage{amsfonts}
\usepackage{verbatim}
\usepackage{latexsym}
\usepackage{ amsthm}
\usepackage{appendix}

\newcommand{\allowpagebreak}
\allowdisplaybreaks

\newtheorem{thm}{Theorem}[section]
\newtheorem{lem}{Lemma}[section]
\newtheorem{cor}{Corollary}[section]
\newtheorem{defi}{Definition}[section]
\newtheorem{assp}{Assumption}[section]
\newtheorem{ex}{Example}[section]
\newtheorem{prop}{Proposition}[section]
\newtheorem{rem}{Remark}[section]
\def\x{\Xi}

 \allowdisplaybreaks
\begin{document}

\begin{frontmatter}



\title{Backward Stochastic Differential Equations with Markov Chains and The Application: Homogenization of PDEs System\fnref{fn2}}

\author[a]{Huaibin Tang\fnref{fn1}}
\ead{tanghuaibin@gmail.com}
\address[a]{School of Mathematics, Shandong University, Jinan, Shandong 250100, P. R. China.}

\author[a]{ Zhen Wu\corref{cor1} }

\ead{wuzhen@sdu.edu.cn}
\cortext[cor1]{Corresponding author. Tel:  86-531-88369577,
                  Fax: 86-531-88365550.}
\fntext[fn1]{Current address: INRIA-IRISA, Campus de Beaulieu, 35042 Rennes Cedex, France.}
\fntext[fn2]{This work was supported by Natural Science Foundations
of China (No. 10671112) and Shandong Province (No. JQ200801 and
2008BS01024), the National Basic Research Program of China (973
Program, No. 2007CB814904) and the Science Fund for Distinguished
Young Scholars of Shandong University (No. 2009JQ004).}

\begin{abstract}
Stemmed from the derivation of the optimal control to a stochastic
linear-quadratic  control problem with Markov jumps, we study one
kind of backward stochastic differential equations (BSDEs) that the
generator  $f$ is affected by a  Markovian switching.
Then, the case that the Markov chain is involved in a large state
space is considered. Following the classical approach, a
hierarchical approach is adopted to reduce the complexity and a
singularly perturbed Markov chain is involved. We will study the
asymptotic property of   BSDE with the singularly perturbed
Markov chain. At last, as an application of our theoretical result,
we show the homogenization of one system of partial differential
equations (PDEs) with a singularly perturbed Markov chain.

\end{abstract}


\end{frontmatter}

\section{Introduction} \label{sec: Introduction}
The study of backward stochastic differential equations (BSDEs in
short) stemmed from stochastic control problem (\cite{Bismut1973}) in
which a non-linear Ricatti   BSDE was introduced.  Then, after the
pioneering work of Pardoux and Peng (\cite{Pardoux1990}) about the
general BSDE which was driven by a Brownian motion, BSDEs have been
extensively studied in the last twenty years because of their connections
with mathematical finance, stochastic control, and partial
differential equations (PDEs in short), please refer to
\cite{Karoui1997, Duffie1992, Duffie1992a}. Since then, many
researchers devoted their work to more general  BSDE, such as BSDE
driven by a L\'{e}vy process (\cite{Bahlali2003, Nualart2001}) and
BSDE with respect to both a Brownian motion and a Poisson random
measure (\cite{Barles1997, Tang1994, Idris2010}). Recently, Cohen and Elliott
\cite{COHEN2008, Cohen2010} studied BSDE driven by the martingale
part of a Markov chain and its application in finance.

As studied in Zhang and Yin (\cite{Zhang2004}), we consider  the stocks investment models
by virtue of hybrid geometric Brownian motion  in which both the
expected return and volatility depend on a finite state Markov
chain. To capture the  market trends as well as the various economic
factors, a finite state Markov chain $\alpha_t$, $t\ge 0$, is introduced to represent the
general market directions. If our object is to allocate assets into a number
of stocks so as to maximize an expected utility within a finite time
horizon, it leads to an optimal control problem.
 By virtue of
maximum principle method, we need to introduce an  adjoint equation
to deal with this optimization problem. The adjoint equation and the
state equation form  forward backward stochastic differential
equations (FBSDEs in short) system. Here the adjoint equation will
be a kind of BSDE with a Markov chain.

Motivated by such adjoint equation, in our paper,
 we consider the following BSDE with a Markov chain:
\[       Y_t=\xi+\int_t^Tf(s,Y_s,
                Z_s,\alpha_s)ds-\int_t^TZ_sdB_s,
\]
where $\alpha=\{\alpha_t;0\leq t\leq T\}$ is a continuous-time
Markov chain independent of the Brownian motion $B$. It is noted that this BSDE is different from the one studied in Cohen and Elliott (\cite{COHEN2008}), and it can be
considered as that its generator  is disturbed by random
environment and takes a set of discrete values which can be
described by a Markov chain.

When study the solvability for BSDE with a Markov chain, the classic
method with the contraction mapping cannot be directly used for the
lack of suitable filtration and corresponding It\^{o}'s
representation theorem. In this paper, inspired by the method
dealing with the BSDE with doubly Brownian motion
(\cite{Pardoux1994}), we construct a new filtration and give the
corresponding  extended It\^{o}'s representation theorem.

When various factors are considered, the underlying Markov chain
inevitably has a large state space, and the corresponding BSDE
becomes increasingly complicated. It is rationale that the change
rates of states display a two-time-scale behavior, a fast-time scale
and a slow varying one. Under this case, a small parameter $
\varepsilon
>0$ can be introduced and the singularly perturbed Markov chain is
involved.

In this paper, we consider the case that the states of the
underlying Markov chain are divided into a number of weakly
irreducible
 classes such that the Markov chain fluctuates rapidly among different states, and jumps less frequently among those classes.
 To reduce the complexity,  a small parameter ($\varepsilon >0$) is introduced to reflect the different rate of changes among
 different states. As shown in Zhang and Yin (\cite{Yin1997}),  it leads to a singularly perturbed Markovian models with two-time
 scale, the actual time $t$ and the stretched time $\frac{t}{\varepsilon}$. By aggregating the states in each irreducible class
  into a single, a limit aggregated Markov chain with a considerably smaller space can be obtained and its asymptotic
  probability distribution is studied. Such asymptotic theory has many applications in optimal control problem and mathematical finance.
We refer interested readers to \cite{Zhang2004,
 Zhang1999}.

In this paper, we will focus on  the asymptotic property of BSDE with a singularly perturbed
Markov chain. Following the averaged approach that aggregating the
states according to their jump rates,  we will show that the distribution of
$(Y_t, \int_t^T Z_sd\bar{B}_s)$ can be seen as an  asymptotic
distribution to  $(Y^{\varepsilon}_t,$  $\int_t^T
Z^{\varepsilon}_sdB_s)$,  where $(Y^{\varepsilon},Z^{\varepsilon})$
and
                $(Y,Z)$ satisfy:
                 \begin{equation*}
                Y^{\varepsilon}_t=\xi+\int_t^Tf(s,Y^{\varepsilon}_s,
                \alpha^{\varepsilon}_s)ds-\int_t^TZ^{\varepsilon}_sdB_s
  \end{equation*}and
                 \begin{equation*}
                Y_t=\xi+\int_t^T\bar{f}(s,Y_s,
                \bar{\alpha}_s)ds-\int_t^TZ_sd\bar{B}_s.
\end{equation*}
Here $\bar{\alpha}$ and $\bar{f}(s,Y_s,\bar{\alpha}_s)$ are respectively the limit
aggregated Markov chain and the averaged generator with respect to the quasi stationary distributions of the singularly perturbed Markov chain. Compared to the original BSDE with the singularly perturbed Markov chain, the limit BSDE depends on a Markov chain with a much smaller state space. Thus the complexity is reduced.

It is well known that BSDEs provide a probabilistic representation
for the solution of a large class of quasi-linear second order
partial differential equations (PDEs in short) (\cite{Barles1997, Pardoux1992, Pardoux1994, Pardoux1997, Ma1994}).
Thus BSDEs provide a probabilistic tool to study the
homogenization of PDEs, which
is the process of replacing rapidly varying coefficients by new
ones thus the solutions are close. In this paper, as an application of our theoretical result, after showing the relation
between our BSDE and one system of semi-linear PDE, we will show the
homogenization result of one system of semi-linear PDE with a singularly perturbed Markov
chain.

 This paper is organized as following. In section \ref{section:BSDE}, we
 study the solvability of BSDE with a Markov
chain.
Section \ref{section: BSDE with sungularly perturbed M C} is devoted
to the case that Markov chain has a large space. Under Jakubowski S-topology (\cite{Jakubowski1997}
) which is weaker than Skorohod's topology, we present the asymptotic property of BSDE with a singularly perturbed Markov chain.
  In section
\ref{section: homogenization of SPDEs},  we show the application of our theoretical results in the homogenization of one system of semi-linear PDE  with a singularly perturbed Markov chain. For  the terseness of the
main text of our paper, we put part of technical proofs for some
results in Appendix.

\section{BSDEs with Markov Chains} \label{section:BSDE}
Let $(\Omega, {\cal{F}},  P)$ be a probability space and $T>0$ be fixed. $\{{\cal{H}}_t,0 \leq t\leq
T\}$ is a filtration on the space satisfying the usual condition.
$B=\{B_t;0\leq t\leq T\}$ with $B_0=0$ is a $d$-dimensional
${\cal{H}}_t$-Brownian motion, and $\alpha=\{\alpha_t;0\leq t\leq
T\}$ is a continuous-time Markov chain independent of $B$ with the
state space ${\cal{M}}= \{1,2, \ldots,m \}$. Suppose the generator
of the Markov chain $Q=(q_{ij})_{m\times m}$ is given by
\begin{equation*}
\begin{split}
   P\{\alpha(t+\triangle)=j|\alpha(t)=i\}= \left\{\begin{array}{l l}
     q_{ij}\triangle+o(\triangle),\ \textrm{ \quad  \ \ if } i \neq j\\
     1+q_{ij}\triangle+o(\triangle)  ,\ \textrm{ if } i = j
         \end{array}
  \right.
\end{split}
\end{equation*}
where $\triangle > 0$. Here $q_{ij}\geq 0$ is the transition rate
from $i$ to $j$ if $i \neq j$, while $ q_{ii}=-\sum_{j=1,i\neq j}^m
q_{ij}$.

Throughout this paper, we introduce the following notations:
 $|\cdot|$ is the norm in the corresponding space; $A'$ is the transpose of matrix $A$;
 $L^p({\cal{H}}_{t};R^n)$ is the space of $R^n$-valued
${\cal{H}}_{t}$-adapted random variable $\xi$ satisfying
$E(|\xi|^p)<\infty$;
  $M_{\mathcal{H}_t}^2(0,T;$ $R^n)$ denotes the space of
$R^n$-valued $\mathcal{H}_t$-adapted  stochastic processes
$\varphi=\{\varphi_t;t \in [0,T]\}$ satisfying
$E\int_0^T|\varphi_t|^2dt < \infty$;
 $S_{\mathcal{H}_t}^2(0,T;R^n)$ is the space of
$R^n$-valued ${\cal{H}}_{t}$-adapted continuous
 stochastic processes $\varphi=\{\varphi_t;t \in [0,T]\}$ satisfying
$E(\sup_{0\leq t\leq T}|\varphi_t|^2)$ $ < \infty$.

\subsection{Motivation}

To study the stochastic optimal control  problem with a Markov
chain, we introduce an adjoint equation, then the state equation and
adjoint equation form a kind of FBSDEs with a Markov chain. Here the
adjoint equation will be a BSDE with  Markov chain. We give the
following linear quadratic (LQ in short) optimal control problem as an
example.

Consider the following stochastic LQ control problem with Markov
jumps
\begin{subequations} \label{LQ problem 1}
\begin{align}
      \textnormal{min.} \
                            &
                            J(v)=\frac{1}{2}E\left(\int_0^T\left((x^v_t)' R(t,\alpha_t)x^v_t+v'_t N(t,\alpha_t) v_t\right)dt+(x^v_T)'Q(\alpha_T)x^v_T\right)
                   \\
   \textnormal{s. t. }  & \left\{\begin{array}{l l}
     dx^v_t=\left(A(t,\alpha_t)x^v_t+B(t,\alpha_t)v_t\right)dt+\left(C(t,\alpha_t)x^v_t+D(t,\alpha_t)v_t\right)dB_t\\               x^v_0=a \in R^n
         \end{array}
  \right.
\end{align}
\end{subequations}
where $A(t,\alpha_t)=A_i(t), B(t,\alpha_t)=B_i(t),
C(t,\alpha_t)=C_i(t), D(t,\alpha_t)=D_i(t),$ $ R(t,\alpha_t)=R_i(t),
N(t,\alpha_t)=N_i(t)$ when $\alpha_t=i\ (i=1,\cdots,m)$, and they
are uniformly bounded $\mathcal{F}_t^B$-adapted processes
with appropriate dimensions. $Q(\alpha_T)=Q_i$ when $\alpha_T=i$
$(i=1,\cdots,m)$, and it is nonnegative symmetric matrices-valued
$\mathcal{F}_T^B$-measurable random variable. Besides, $R_i(t)$ is
nonnegative symmetric matrices-valued, $N_{i}(t)$ is positive
symmetric matrices-valued and the inverse $N_{i}(t)^{-1}$ is
bounded. The set of all $ \mathcal{H}_t$-adapted admissible controls is
$\mathcal{U}_{ad}\equiv M^2_{\mathcal{H}_t}(0,T;R^{n_u \times d})$,
and our aim is to find an admissible control $u$ such that
$\displaystyle J(u)=\inf_{v \in \mathcal{U}_{ad}}J(v)$.

There are many literatures on this kind of LQ optimal control problem with
Markov jumps \eqref{LQ problem 1} and its application, such as \cite{Zhou2003, Li2002, Zhang1999} and their references. Different to their methods that constructing the optimal control via the solution of Riccati equation,  we will use the FBSDEs approach.

\begin{thm} \label{thm: LQ problem and fbsde}
If the following FBSDE admits a unique solution $(x_t, y_t, z_t)$
\begin{equation} \label{fbsde of LQ problem}
\begin{split}
       \left\{\begin{array}{l l l}
                   ~~ dx_t=\left(A(t,\alpha_t)x_t+B(t,\alpha_t)\left( -N^{-1}(t,\alpha_t)\left( B'(t,\alpha_t)y_t+D'(t,\alpha_t)z_t\right)\right)\right)dt\\
                          ~~~~+\left(C(t,\alpha_t)x_t+D(t,\alpha_t)\left( -N^{-1}(t,\alpha_t)\left( B'(t,\alpha_t)y_t+D'(t,\alpha_t)z_t\right)\right)\right)dB_t,\\
                    -dy_t=\left(A'(t,\alpha_t)y_t+C'(t,\alpha_t)z_t+R(t,\alpha_t)x_t\right)dt-z_t dB_t,\\
                   ~~~ x_0=a, \ \ y_T=Q(\alpha_T)x_T.
         \end{array}
  \right.
\end{split}
\end{equation}
Then
$$u_t = -N^{-1}(t,\alpha_t)\left(B'(t,\alpha_t)y_t+D'(t,\alpha_t)z_t\right), 0\le t\le T$$
is the unique optimal control for the LQ problem \eqref{LQ problem 1}
\end{thm}
\begin{proof} Firstly, we will prove that $\{u=u_t; 0\le t\le T\}$ is an optimal
control  for the LQ problem \eqref{LQ problem 1}.

 From the forward equation of \eqref{fbsde of LQ problem}, we can see that $x$ is the corresponding system state trajectory of $u$. For an arbitrary admissible control $v$, denote $x^v$ as the corresponding system state trajectory, then
  \begin{align*}
              & J(v)-J(u)\\
            = & \ \frac{1}{2} E \Big( \int_0^T \big((x_t^v-x_t)'R(t,\alpha_t)(x_t^v-x_t)+(v_t-u_t)'N(t,\alpha_t)(v_t-u_t) \\
              & \ \ \  \ \ \ \ \ \ \ \ \ \ \ +2x_t'R(t,\alpha_t)(x_t^v-x_t) +2u_t'N(t,\alpha_t)(v_t-u_t)   \big)dt
              \\
              & \ \ \  \ \ \ \ +(x_T^v-x_T)'Q(\alpha_T)(x_T^v-x_T)+ 2x'_TQ(\alpha_T)(x_T^v-x_T) \Big).
  \end{align*}
Applying  It\^{o}'s formula to $y'_t(x_t^v-x_t)$, we have
\begin{align*}
             & E x'_TQ(\alpha_T)(x_T^v-x_T)\\
             =&E \int_0^T\big( \left ( y'_t B(t,\alpha_t) + z'_t D(t,\alpha_t)\right)(v_t+N^{-1}(t,\alpha_t)\left(B'(t,\alpha_t)y_t+D'(t,\alpha_t)z_t\right))\\
            &~~~~~~~~~~ -x_t'R(t,\alpha_t)(x_t^v-x_t) \big)dt\\
             = & E \int_0^T\left( \left ( y'_t B(t,\alpha_t) + z'_t D(t,\alpha_t)\right)(v_t-u_t)-x_t'R(t,\alpha_t)(x_t^v-x_t) \right)dt.
\end{align*}
As $R$, $Q$ are nonnegative and $N$ is positive, we have
\begin{align*}
              J(v)-J(u) \geq & E \int_0^T \left(y'_t B(t,\alpha_t) + z'_t D(t,\alpha_t)+u'_t N(t,\alpha_t)\right)(v_t-u_t) dt \\
                      = & 0.
\end{align*}
So $u(t)=-N^{-1}(t,\alpha_t)\left(B'(t,\alpha_t)y_t+D'(t,\alpha_t)z_t\right)$
is an optimal control.

{\bf{Uniqueness:}} Assume that $u^1$ and $u^2$ are both optimal
controls with $J(u^1)=J(u^2)=\gamma \geq 0$, and the corresponding
trajectories are $x^1$ and $x^2$. Due to the linear property of the
system, the trajectories corresponding to $\displaystyle
\frac{u^1+u^2}{2}$ is $\displaystyle \frac{x^1+x^2}{2}$.

From the classical parallelogram rule,  $R$, $Q$ are
nonnegative, and $N$ is positive, there exists $\delta >0$ such that
\begin{align*}
              &2 \gamma\\
              = & \ J(u^1)+J(u^2)\\
              = & \ 2J\left(\frac{u^1+u^2}{2}\right) + E \bigg(
                \left(\frac{x_T^1-x_T^2}{2}\right)'Q(\alpha_T)\left(\frac{x_T^1-x_T^2}{2}\right) + \int_0^T \bigg( \left(\frac{x_t^1-x_t^2}{2}\right)' \\
               &\ \ \ \ \  \ \ \ \ \ \  R(t,\alpha_t)\left(\frac{x_t^1-x_t^2}{2}\right) +
              \left(\frac{u_t^1-u_t^2}{2}\right)'N(t,\alpha_t)\left(\frac{u_t^1-u_t^2}{2}\right)\bigg)dt \bigg)\\
              \geq & \ 2J\left(\frac{u^1+u^2}{2}\right) + E\int_0^T  \left(\frac{u_t^1-u_t^2}{2}\right)'N(t,\alpha_t)\left(\frac{u_t^1-u_t^2}{2}\right) dt \\
               \geq & \ 2\gamma +\frac{\delta}{4}E\int_0^T |u_t^1-u_t^2|^2dt
\end{align*}
Thus $E\int_0^T |u_t^1-u_t^2|^2dt \leq 0$ which yields that $u^1=u^2$.
\end{proof}

\subsection{BSDEs with Markov chains}

The derivation of the optimal control for the above LQ problem \eqref{LQ problem 1} can
be regarded as one motivation for us to study BSDEs with Markov
jumps.  In this subsection, we will study the solvability to the
 following BSDE with a Markov chain firstly:
\begin{equation} \label{eq:BSDE}
                Y_t=\xi+\int_t^Tf(s,Y_s,
                Z_s,\alpha_s)ds-\int_t^TZ_sdB_s.
\end{equation}

Let ${\cal{N}}$ denote the class of all $P$-null sets of
${\cal{F}}$. For each $t \in [0,T]$, we define
${\cal{F}}_t={\cal{F}}_t^B \vee{\cal{F}}_{t,T}^{\alpha}
\vee{\cal{N}}$ where for any process $\{\eta_t; 0 \leq t \leq T\}$,
${\cal{F}}_{t,T}^{\eta}=\sigma \{\eta_r;t \leq r\leq T\}$ and
${\cal{F}}_{t}^{\eta}={\cal{F}}_{0,t}^{\eta}$.
For convenience, we denote $M^2(0,T;R^n)=M_{\mathcal{F}_t}^2(0,T;R^n)$ and $S^2(0,T;$ $R^n)=$ $S_{\mathcal{F}_t}^2(0,$ $T;R^n)$.

Thereinafter,    we make the following assumption:

\begin{assp} \label{assp:f of BSDE}
               (i)  $\xi \in L^2({\cal{F}}_{T};R^k)$;
                (ii) $f:\Omega\times     [0,T]\times R^k \times R^{k\times d} \times\mathcal{M}\rightarrow  R^k$ satisfies that $
                \forall  (y,z)\in R^k \times R^{k\times d}$,
                   $ \forall i \in \mathcal{M}$, $f(\cdot, y, z, i)  \in M^2_{\mathcal{F}_t^B}(0,T;R^k)$, and
                             $\exists\mu>0$, such that  $ \forall i \in \mathcal{M}$, $ \forall  (\omega, t) \in \Omega\times
                 [0,T]$, $(y_1,z_1)$, $(y_2,z_2) \in R^k \times R^{k\times d}$,
                  $$|f(t,y_1,z_1,i)-f(t,y_2,z_2,i)| \leq                    \mu(|y_1-y_2|+|z_1-z_2|).                $$
\end{assp}
 Our main result in this section is in the following theorem.
\begin{thm} \label{Thm: existence of the solution of BSDE}
                Under Assumption $\ref{assp:f of BSDE}$, there exists a unique solution
                pair $(Y,Z)\in S^2(0,T;R^k)\times M^2(0,T;R^{k\times d})$ for BSDE $(\ref{eq:BSDE})$.
\end{thm}
The proof of Theorem \ref{Thm: existence of the solution of BSDE} consists of three steps.

 \noindent {\bf{Step 1: Extension of It\^{o}'s  representation theorem}}

 It is noted that
$\{{\cal{F}}_t; 0 \leq t \leq T\}$ is neither increasing nor
decreasing, and it does not constitute a filtration. Inspired by the method handling the BSDE with doubly Brown motions (\cite{Pardoux1994}),  we define a filtration $({\mathcal{G}}_t)_{0\leq t\leq T}$ by
$${\cal{G}}_t\triangleq{\cal{F}}_t^B\vee {\cal{F}}_T^{\alpha}\vee \cal{N}$$
For the filtration $({\mathcal{G}}_t)_{0\leq t\leq T}$, we give the
following  extension of It\^{o}'s  representation theorem. This
result and its corollary play key roles during the proof of Theorem
\ref{Thm: existence of the solution of BSDE}.

\begin{prop} \label{Prop: extension of ito martingale}
               For $N\in L^2(\mathcal{G}_T;R^k)$,
              there exist a unique random variable $N_0\in L^2(\mathcal{F}_T^{\alpha};R^k)$ and
              a unique stochastic process $Z=\{Z_t; 0\leq t\leq T\} \in M^2_{\mathcal{G}_t}(0,T; R^{k \times d})$  such that
              \begin{equation}\label{eq: extension of Ito representation}
                              N=N_0+\int_0^TZ_tdB_t, \quad 0\leq t\leq T.
              \end{equation}
              Actually, $N_0=E(N|\mathcal{F}_T^{\alpha})$.
\end{prop}
During the derivation of Proposition \ref{Prop: extension of ito martingale}, we need the following two lemmas.
 \begin{lem} (\cite{Yan2003})\label{Lemma:biaoshi lemma}
       If $X$, $Y$: $\Omega \rightarrow R^d$ are two given
       functions,  $Y$ is $\sigma(X)$-measurable if and only if
       there exists a Borel measurable function $g$: $R^d \rightarrow
       R^d$ such that $Y=g(X)$.
 \end{lem}
 \begin{lem}(Doob's martingale convergence theorem)\label{Doob's martingale convergence theorem}
            Let $\{\mathcal{F}_t; t\geq 0\}$ be a filtration on the space $(\Omega, \mathcal{F}, P)$, $X \in L^1(\mathcal{F};R^d)$, then
            $$
              E(X|\mathcal{F}_t)\rightarrow
              E(X|\mathcal{F}_{\infty}), \quad \mbox{as } t \rightarrow
              \infty, \textrm{      a.s. and in $L^1$ as well.}
            $$

\end{lem}
\begin{proof}
\noindent {{\bf Existence:}} Let $\{t_i\}_{i\geq 0}$,
$\{t'_j\}_{j\geq 0}$ be two dense subsets of $[0,T]$ where $t_0=t'_0=0$. For each integer $n, m \geq
 0$, let $\mathcal{G}_{n,m}$ be the $\sigma$-algebra generated by $\alpha_{t_0}, \alpha_{t_1}, \cdots, \alpha_{t_n},
 B_{t'_0},B_{t'_1},$ $\cdots,B_{t'_m}$, i.e., $\mathcal{G}_{n,m}=\sigma\{\alpha_{t_0}, \alpha_{t_1}, \cdots, \alpha_{t_n},
 B_{t'_0},B_{t'_1},$ $\cdots,B_{t'_m}\}$. Obviously,  $\mathcal{G}_{n,m} \subset
 \mathcal{G}_{n+1,m}$, $\mathcal{G}_{n,m} \subset
 \mathcal{G}_{n,m+1}$, $\mathcal{G}_{n,m} \subset
 \mathcal{G}_{n+1,m+1}$, $\displaystyle \sigma(\cup_{m=1}^{\infty}\mathcal{G}_{n,m})=\mathcal{F}_T^B\vee \sigma\{\alpha_{t_0}, \alpha_{t_1}, \cdots, \alpha_{t_n}\},$ and $\displaystyle \sigma(\cup_{n,m=0}^{\infty}$ $\mathcal{G}_{n,m})$ $=\mathcal{G}_T$.

 For $N \in L^2(\mathcal{G}_T;R^k)$, $\forall n, m$, by Lemma  \ref{Lemma:biaoshi lemma} and Lemma \ref{Doob's martingale convergence theorem}, there exists a Borel measurable function $N_{n,m}: \mathcal{M}^{n+1} \times R^{(m+1) \times d} \rightarrow
 R^k$ such that
 \begin{equation*}
 \begin{split}
             &E(N|\mathcal{G}_{n,m})=N_{n,m}(\alpha_{t_0}, \alpha_{t_1}, \cdots, \alpha_{t_n}, B_{t'_0},B_{t'_1},\cdots,B_{t'_m})\\
             &N_{n,m}(\alpha_{t_0}, \alpha_{t_1}, \cdots, \alpha_{t_n}, B_{t'_0},B_{t'_1},\cdots,B_{t'_m}) \rightarrow
              E\left(N|\mathcal{F}_T^B\vee \sigma\{\alpha_{t_0}, \alpha_{t_1}, \cdots, \alpha_{t_n}\}\right)
 \end{split}
 \end{equation*}
 Denote $ N_n(\alpha_{t_0}, \alpha_{t_1}, \cdots, \alpha_{t_n})
             \triangleq  E\left(N|\mathcal{F}_T^B\vee \sigma\{\alpha_{t_0}, \alpha_{t_1}, \cdots, \alpha_{t_n}\}\right)
            $, it can be rewritten as
 $$
        N_n(\alpha_{t_0}, \alpha_{t_1}, \cdots, \alpha_{t_n})=
        \sum_{i_0,i_1,\cdots,i_n =1}^m I_{\{(\alpha_{t_0}, \alpha_{t_1}, \cdots,
        \alpha_{t_n})=(i_0,i_1,\cdots,i_n)\}}N_n(i_0,i_1,\cdots,i_n)
$$
where $N_n(i_0,i_1,\cdots,i_n)$ is $\mathcal{F}_T^B$-measurable.

For $ (i_0,i_1,\cdots,i_n) \in \mathcal{M}^{n+1}$, applying It\^o's representation theorem,
$$
            N_n(i_0,i_1,\cdots,i_n )  = N_0{(i_0,i_1,\cdots,i_n)} +\int_0^TZ_t{(i_0,i_1,\cdots,i_n)}dB_t
$$
 where   $N_0(i_0,i_1,$ $\cdots,i_n)$ is a constant and  $Z{(i_0,i_1,\cdots,i_n)} \in M^2_{\mathcal{F}_t^B}(0,T;R^{k \times d})$.

 Denote \newline $\displaystyle N_0(\alpha_{t_0}, \alpha_{t_1}, \cdots,  \alpha_{t_n}) = \sum_{i_0,i_1,\cdots,i_n =1}^m I_{\{(\alpha_{t_0}, \alpha_{t_1}, \cdots, \alpha_{t_n})=(i_0,i_1,\cdots,i_n)\}} N_0{(i_0,i_1,\cdots,i_n)}$, and
 $\displaystyle Z_t(\alpha_{t_0}, \alpha_{t_1}, \cdots,\alpha_{t_n}) = \sum_{i_0,i_1,\cdots,i_n =1}^m I_{\{(\alpha_{t_0}, \alpha_{t_1}, \cdots, \alpha_{t_n})=(i_0,i_1,\cdots,i_n)\}} Z_t{(i_0,i_1,\cdots,i_n)}$. Clearly, $N_0(\alpha_{t_0}, $ $\alpha_{t_1}, \cdots,
   \alpha_{t_n})$ $ \in   L^2(\mathcal{F}_T^{\alpha};R^k)$, $\{Z_t(\alpha_{t_0}, $ $ \alpha_{t_1}, \cdots,
   \alpha_{t_n});0\leq t\leq T\} \in M^2_{\mathcal{G}_t}(0,T;R^{k\times d})$, and we have
   \begin{eqnarray} \label{eq1: formula 1 in extension of Ito}
                      N_n(\alpha_{t_0}, \alpha_{t_1}, \cdots, \alpha_{t_n}) = N_0(\alpha_{t_0}, \alpha_{t_1}, \cdots,  \alpha_{t_n})
                      +\int_0^T Z_t(\alpha_{t_0}, \alpha_{t_1}, \cdots,\alpha_{t_n})dB_t.
   \end{eqnarray}

In the remained part, we will prove that as $n \rightarrow \infty$,  both side of \eqref{eq1: formula 1 in extension of Ito}    are Cauchy sequences.

For the left hand side, as $n \rightarrow \infty$, with the definition of $ N_n(\alpha_{t_0}, \alpha_{t_1}, \cdots, \alpha_{t_n})$ and Lemma \ref{Doob's martingale convergence theorem}, we have
\[
       N_n(\alpha_{t_0}, \alpha_{t_1}, \cdots, \alpha_{t_n}) \rightarrow N.
\]
For the right hand side,  since
   \begin{eqnarray} \label{eq2: formula 2 in extension of Ito}
                                          && E\left(\int_0^T Z(\alpha_{t_0}, \alpha_{t_1}, \cdots, \alpha_{t_n})dB_t|\mathcal{F}_T^{\alpha}\right)                                          \nonumber \\ &=&\sum_{i_0,i_1,\cdots,i_n =1}^m I_{\{(\alpha_{t_0}, \alpha_{t_1}, \cdots,
                                     \alpha_{t_n})=(i_0,i_1,\cdots,i_n)\}}  E\left(\int_0^T Z{(i_0,i_1,\cdots,i_n)}
                                     dB_t|\mathcal{F}_T^{\alpha}\right)\nonumber \\  &=& 0.
\end{eqnarray}
 Thus $N_0(\alpha_{t_0}, \alpha_{t_1}, \cdots, \alpha_{t_n})=E(N_n(\alpha_{t_0}, \alpha_{t_1}, \cdots, \alpha_{t_n})|\mathcal{F}_T^{\alpha})$. As $n\rightarrow   \infty$, we can conclude that
 \[
      N_0(\alpha_{t_0}, \alpha_{t_1}, \cdots,\alpha_{t_n})    \rightarrow \  E(N|\mathcal{F}_T^{\alpha}).
 \]

 Now let us consider the sequence $\{Z_t(\alpha_{t_0}, \alpha_{t_1}, \cdots, \alpha_{t_n});0\leq t\leq T\}$. For $n,m \geq 0$,
 \allowdisplaybreaks {
 \begin{equation*}
 \begin{split}
                   & E\int_0^T\left|Z_t(\alpha_{t_0}, \alpha_{t_1}, \cdots,  \alpha_{t_n})-Z_t(\alpha_{t_0}, \alpha_{t_1}, \cdots,
                      \alpha_{t_m})\right|^2dt \\
                     = \ & E\left(\int_0^T\left(Z_t(\alpha_{t_0}, \alpha_{t_1}, \cdots,  \alpha_{t_n})-Z_t(\alpha_{t_0}, \alpha_{t_1},
                     \cdots, \alpha_{t_m})\right)dB_t\right)^2\\
                  =\ & E\big(E\left(N|\mathcal{F}_T^B\vee \sigma\{\alpha_{t_0}, \alpha_{t_1}, \cdots,
                  \alpha_{t_n}\}\right)-E\left(N|\mathcal{F}_T^B\vee \sigma\{\alpha_{t_0}, \alpha_{t_1}, \cdots,
                                 \alpha_{t_m}\}\right)\\
                  &-N_0(\alpha_{t_0}, \alpha_{t_1}, \cdots, \alpha_{t_n})+N_0(\alpha_{t_0}, \alpha_{t_1}, \cdots,
                  \alpha_{t_m}) \big)^2\\
              \leq \ & 2 E\big(E(N|\mathcal{F}_T^B\vee \sigma\{\alpha_{t_0}, \alpha_{t_1}, \cdots,
              \alpha_{t_n}\})^2-E(N|\mathcal{F}_T^B\vee \sigma\{\alpha_{t_0}, \alpha_{t_1}, \cdots,
              \alpha_{t_m}\})\big)^2\\
                &+ 2E(N_0(\alpha_{t_0}, \alpha_{t_1}, \cdots,  \alpha_{t_n})-N_0(\alpha_{t_0}, \alpha_{t_1}, \cdots,
                \alpha_{t_m}) \big)^2 \\
              \rightarrow \ & 0,  \quad \textrm{ as } n,m \rightarrow \infty.
   \end{split}
 \end{equation*}}

Thus $\{Z_t(\alpha_{t_0}, \alpha_{t_1}, \cdots, \alpha_{t_n});0\leq t\leq T\}$ is a Cauchy sequence in $M^2_{\mathcal{G}_t}(0,T;$ $R^{k \times  d})$. Hence it converges to some $Z \in M^2_{\mathcal{G}_t}(0,T;R^{k \times   d})$.

Denote $N_0 = E(N|\mathcal{F}_T^{\alpha})$, we can conclude that   $(\ref{eq1: formula 1 in extension of Ito})$ converges to the extended  It\^o's representation \eqref{eq: extension of Ito representation}.

 \noindent{{\bf Uniqueness:}}  By virtue of equation \eqref{eq2: formula 2 in extension of Ito} and the fact that  as $n \rightarrow \infty$,   $\{Z_t(\alpha_{t_0}, \alpha_{t_1}, \cdots,
   \alpha_{t_n});$ $0\leq t\leq T\}$ is a Cauchy sequence,
   we have $E(\int_0^TZ_tdB_t|\mathcal{F}_T^{\alpha})$ $=0$.
   Then, for $(N_0,Z), (N_0', Z')$ satisfying the extended It\^o's representation $(\ref{eq: extension of Ito
   representation})$,  we get $N_0=N'_0$ by taking conditional expectation with respect to
   $\mathcal{F}_T^{\alpha}$. Uniqueness of $Z$ follows easily from the fact that
 $$E\int_0^T|Z_t-Z'_t|^2dt=E\left(\int_0^T(Z_t-Z'_t)dB_t\right)^2
 =E\left(N_0-N_0'\right)^2  =0. $$
 \end{proof}

 The following corollary is useful in the proof of Theorem \ref{Thm: existence of the solution of BSDE} and its proof is similar to Proposition
\ref{Prop: extension of ito martingale}.
\begin{cor} \label{coro:martingale representation}
              For $t\leq T$, we consider the filtration $(\mathcal{N}_s)_{t\leq s\leq T}$
              defined by $\mathcal{N}_s=\mathcal{F}_s^B\vee
              \mathcal{F}_{t,T}^{\alpha}$. For $N\in L^2(\mathcal{N}_T;R^k)$,
              there exists  
              a unique stochastic process $Z=\{Z_s; t \leq s\leq T\} \in M^2_{\mathcal{N}_s}(t,T; R^{k \times d})$  such that
              $$N=E(N|\mathcal{N}_t)+\int_t^TZ_sdB_s.$$
 \end{cor}

\noindent {\bf{Step 2: The special case: the generator $f$ is independent of $y$ and $z$.}}
\begin{prop} \label{Prop:f independent of y and z}
              Under Assumption $\ref{assp:f of BSDE}$, the following
              BSDE
             \begin{equation} \label{eq:BSDE f independent of y and z}
                           Y_t=\xi+\int_t^Tf(s,\alpha_s)ds-\int_t^TZ_sdB_s,
                           \qquad 0\leq t\leq T
              \end{equation}
             has a  solution pair   $(Y,Z)\in S^2(0,T;R^k)\times M^2(0,T;R^{k\times d})$.
\end{prop}
\begin{proof}
             From Assumption $\ref{assp:f of BSDE}$ and H\"{o}lder  inequality, we obtain
              $$
                  E\left(\int_0^Tf(s,\alpha_s)ds\right)^2\leq C E\int_0^T |f(s,\alpha_s)|^2ds \leq C \sum_{i=1}^m E\int_0^T|f(s,i)|^2ds                     <\infty
              $$
              which yields
              $$\xi+\int_0^Tf (s,\alpha_s)ds  \in L^2({{\cal{G}}_T};R^k)$$
              For the filtration  $({\mathcal{G}}_t)_{0\leq t\leq T}$ where ${\cal{G}}_t=
              {\cal{F}}_t^B\vee {\cal{F}}_T^{\alpha}\vee {\cal{N}}= {\cal{F}}_{t}^{B} \vee {\cal{F}}_{t,T}^{\alpha}\vee
              {\cal{F}}_{t}^{\alpha}\vee {\cal{N}}  ={\cal{F}}_t \vee {\cal{F}}_{t}^{\alpha}, $
              we can define the following ${\cal{G}}_t$-measurable square integrable martingale
              $$
                N_t=E\left(\xi+\int_0^Tf(s,\alpha_s)ds|{{\cal{G}}_t}\right), \qquad 0\leq t\leq T.
              $$
               By the extended It\^{o}'s representation theorem (Proposition \ref{Prop: extension of ito martingale}), there exist
             $N_0\in L^2(\mathcal{F}_T^{\alpha};R^k)$ and  $Z=\{Z_t; 0\leq t\leq T\} \in M^2_{{\cal{G}}_t}(0,T;R^{k\times
              d})$  such that
              $$ N_t=N_0+\int_0^tZ_sdB_s,\qquad 0\leq t\leq T.$$
              For $t \in [0,T]$, we define
                \begin{equation} \label{Eq:adaptness of Y}
                \begin{split}
                               Y_t= N_t-\int_0^tf(s,\alpha_s)ds,       \textrm{ \ i.e., } Y_t
                                  =E\left(\xi+\int_t^Tf(s,\alpha_s)ds|{{\cal{G}}_t}\right).
               \end{split}
                \end{equation}
             It is easy to verify  that the ${\cal{G}}_t$-measurable process $(Y,Z)$  satisfies BSDE $(\ref{eq:BSDE f independent of y and
              z})$ and $Y \in M^2_{{\cal{G}}_t}(0,T;$ $ R^{k})$. We refer interested reader to \cite{Pardoux1990, Pardoux1999} for the detailed verification.

              The left work is to show that the processes $Y=\{Y_t;0\leq t\leq T\}$ and $Z=\{Z_t;0\leq t\leq T\}$
              are ${\cal{F}}_t$-measurable, i.e. ${\cal{F}}_t^B\vee  {\cal{F}}_{t,T}^{\alpha}$-measurable. $\forall t \in [0,T]$, we denote
              $\vartheta = \xi+\int_t^Tf(s,\alpha_s)ds$,  $\vartheta$ is ${\cal{F}}_T^B\vee
              {\cal{F}}_{t,T}^{\alpha}$-measurable.

               Let $\{\bar{t}_i\}_{i\geq 0}$,  $\{\bar{t}'_j\}_{j\geq 0}$ be respectively dense subsets of $[t,T]$ and $[0,T]$, with
               $\bar{t}_0=t$ and $\bar{t}'_0=0$. For each integer $n, m \geq 0$, let $\bar{\mathcal{G}}_{n,m}$ be the $\sigma$-algebra
               generated by $\alpha_{\bar{t}_0}, \alpha_{\bar{t}_1}, \cdots, \alpha_{\bar{t}_n},
               B_{\bar{t}'_0},B_{\bar{t}'_1},\cdots,B_{\bar{t}'_m}$, i.e., $\bar{\mathcal{G}}_{n,m}=\sigma\{\alpha_{\bar{t}_0},
               \alpha_{\bar{t}_1}, \cdots, \alpha_{\bar{t}_n},$ $ B_{\bar{t}'_0},B_{\bar{t}'_1},\cdots,B_{\bar{t}'_m}\}$.
               Obviously, $\bar{\mathcal{G}}_{n,m} \subset \bar{\mathcal{G}}_{n+1,m+1}$, $\bar{\mathcal{G}}_{n,m} \subset \bar{\mathcal{G}}_{n+1,m}$,
               $\bar{\mathcal{G}}_{n,m} \subset \bar{\mathcal{G}}_{n,m+1}$, and $\sigma(\cup_{n,m=0}^{\infty}\bar{\mathcal{G}}_{n,m})={\cal{F}}_T^B\vee
              {\cal{F}}_{t,T}^{\alpha}$.

                From Lemma \ref{Lemma:biaoshi lemma}, for each $n,m$, there
                exists a Borel measurable function $\vartheta_{nm}:\mathcal{M}^{n+1} \times R^{(m+1)\times d} \rightarrow
                R^k$ such that
              \begin{equation*}
                  \begin{split}
                     E[\vartheta|\bar{\mathcal{G}}_{n,m}]&=\vartheta_{nm}(\alpha_{\bar{t}_0}, \alpha_{\bar{t}_1}, \cdots,
                     \alpha_{\bar{t}_n},                  B_{\bar{t}'_0},B_{\bar{t}'_1},\cdots,B_{\bar{t}'_m}).
                  \end{split}
              \end{equation*}
Since $I_{\{(\alpha_{\bar{t}_0}, \alpha_{\bar{t}_1}, \cdots,
 \alpha_{\bar{t}_n})=(i_0,i_1,\cdots,i_n)\}} \in {\cal{F}}_{t,T}^{\alpha} \subset
 {\cal{F}}_{T}^{\alpha}$, we have
 \begin{equation*}
                  \begin{split}
                                & E\left(\vartheta_{nm}(\alpha_{\bar{t}_0}, \alpha_{\bar{t}_1}, \cdots, \alpha_{\bar{t}_n},
                                     B_{\bar{t}'_0},B_{\bar{t}'_1},\cdots,B_{\bar{t}'_m})|{\cal{F}}_t^B\vee
                                     {\cal{F}}_{T}^{\alpha}\right)\\
                     =& \ E\Bigg(\sum_{i_0,i_1,\cdots,i_n=1}^m I_{\{(\alpha_{\bar{t}_0}, \alpha_{\bar{t}_1}, \cdots,
 \alpha_{\bar{t}_n})=(i_0,i_1,\cdots,i_n)\}} \\
                      &\hspace{1cm} \vartheta_{nm}(i_0,i_1,\cdots,i_n,
                  B_{\bar{t}'_0},B_{\bar{t}'_1},\cdots,B_{\bar{t}'_m})|{\cal{F}}_t^B\vee
              {\cal{F}}_{T}^{\alpha}\Bigg)\\
              =& \sum_{i_0,i_1,\cdots,i_n =1}^m I_{\{(\alpha_{\bar{t}_0}, \alpha_{\bar{t}_1}, \cdots,
 \alpha_{\bar{t}_n})=(i_0,i_1,\cdots,i_n)\}} \\
                &\hspace{1cm}     E(\vartheta_{nm}(i_0,i_1,\cdots,i_n,
                  B_{\bar{t}'_0},B_{\bar{t}'_1},\cdots,B_{\bar{t}'_m})|{\cal{F}}_t^B\vee
              {\cal{F}}_{T}^{\alpha}).
                    \end{split}
    \end{equation*}
      For the reason that  $\vartheta_{nm}(i_0,i_1,\cdots,i_n,
             B_{\bar{t}'_0},B_{\bar{t}'_1},\cdots,B_{\bar{t}'_m})$  is
             $ \mathcal{F}_T^B$-measurable, by It\^o's
             representation theorem, we know that there exist $\nu_0^{n,m}(i_0,i_1,\cdots,i_n) $ $\in
             R^k$ and $Z^{n,m}(i_0,i_1,\cdots,i_n) \in $ $L^2_{\mathcal{F}_t^B}(0,T;$ $R^{k \times d})
             $ such that
             \begin{equation*}
                  \begin{split}
                         &E\left(\vartheta_{nm}(i_0,i_1,\cdots,i_n,
                  B_{\bar{t}'_0},B_{\bar{t}'_1},\cdots,B_{\bar{t}'_m})|{\cal{F}}_t^B\vee
              {\cal{F}}_{T}^{\alpha}\right)\\
              =&E\left(\nu_0^{n,m}(i_0,i_1,\cdots,i_n)+\int_0^TZ_r^{n,m}(i_0,i_1,\cdots,i_n)dB_r|{\cal{F}}_t^B\vee
              {\cal{F}}_{T}^{\alpha}\right)\\
              =&\nu_0^{n,m}(i_0,i_1,\cdots,i_n)+\int_0^tZ_r^{n,m}(i_0,i_1,\cdots,i_n)dB_r
                  \end{split}
    \end{equation*}
     is ${\cal{F}}_t^B$-measurable. Therefore
              \begin{equation*}
                  \begin{split}
                  &E\left(\vartheta_{nm}(\alpha_{\bar{t}_0}, \alpha_{\bar{t}_1}, \cdots,
 \alpha_{\bar{t}_n},
                  B_{\bar{t}'_0},B_{\bar{t}'_1},\cdots,B_{\bar{t}'_m})|{\cal{F}}_t^B\vee
              {\cal{F}}_{T}^{\alpha}\right)\\
              = & \sum_{i_0,i_1,\cdots,i_n =1}^m I_{\{(\alpha_{\bar{t}_0}, \alpha_{\bar{t}_1}, \cdots,
 \alpha_{\bar{t}_n})=(i_0,i_1,\cdots,i_n)\}}\\
                                     & \hspace{1cm} E\left(\vartheta_{nm}(i_0,i_1,\cdots,i_n,
                       B_{\bar{t}'_0},B_{\bar{t}'_1},\cdots,B_{\bar{t}'_m})|{\cal{F}}_t^B\vee
              {\cal{F}}_{T}^{\alpha}\right)       \end{split}
    \end{equation*}
      is ${\cal{F}}_t^B\vee     {\cal{F}}_{t,T}^{\alpha}$-measurable. With Lemma \ref{Doob's martingale convergence theorem},   as $ n,m\rightarrow \infty$,
      $$
                    E[\vartheta|\bar{\mathcal{G}}_{n,m}]\rightarrow E[\vartheta|{\cal{F}}_T^B\vee {\cal{F}}_{t,T}^{\alpha}]=\vartheta.
      $$
      Thus
      $ E[\vartheta|{\cal{F}}_t^B\vee {\cal{F}}_{T}^{\alpha}]$, i.e., $Y_t$, is also ${\cal{F}}_t^B\vee      {\cal{F}}_{t,T}^{\alpha}$-measurable.

             Considering
              $$\int_t^TZ_sdB_s=-Y_t+\xi+\int_t^Tf(s,\alpha_s)ds,$$
              its right side is ${\cal{F}}_T^B
              \vee{\cal{F}}_{t,T}^{\alpha}$-measurable. With Corollary $\ref{coro:martingale representation}$,
              we know  $ \forall  t<s$, $Z_s$ is ${\cal{F}}_s^B\vee
              {\cal{F}}_{t,T}^{\alpha}$-measurable.  Then, by the continuous property of the Markov chain $\alpha$, we obtain that $Z_s$
              is ${\cal{F}}_s^B\vee
              {\cal{F}}_{s,T}^{\alpha}$-measurable.

             Together with  the Burkholder-Davis-Gundy inequality and the form of BSDE $(\ref{eq:BSDE f independent of y and z})$, we can conclude that $\{Y_t;0\leq t\leq
              T\}$ is continuous and satisfies  $\displaystyle E(\sup_{0\leq t\leq T}|Y_t|^2)<
              \infty$. It yields that $Y \in S^2(0,T;R^k)$.
    \end{proof}
  \noindent {\bf{Step 3: The general case: Proof of Theorem {\ref{Thm: existence of the solution
of BSDE}}.}}
\begin{proof}
             Firstly, we    define a mapping $I$ from $M^2(0,T;R^k \times R^{k\times d})$ into itself such that $(Y,Z)\in
             S^2(0,T;R^k)\times M^2(0,T;R^{k\times d})$ is the solution to BSDE $(\ref{eq:BSDE})$ iff it is a fixed point of $I$.

             For a constant $\beta>0$, we introduce the following  equivalent  norm of $M^2(0,T;R^k \times R^{k\times d})$
              $$ \|v(\cdot)\|_{\beta}=\left(E\int_0^T|v_s|^2e^{\beta s}ds\right)^{\frac 1 2}.$$
              For $(y,z)\in M^2(0,T;R^k \times R^{k\times d})$, we set
              $$
                 Y_t=\xi
                 +\int_t^Tf(s,y_s,z_s,\alpha_s)ds-\int_t^TZ_sdB_s.
              $$
              From Assumption $\ref{assp:f of BSDE}$ and H$\ddot{\textrm{o}}$lder's inequality,
               \begin{equation*}
                               \begin{split}
                                              &E\left(\int_0^Tf(s,y_s,z_s,\alpha_s)ds\right)^2
                                              \\ \leq        &\                     2E\left(\int_0^T(f(s,y_s,z_s,\alpha_s)-f(s,0,0,\alpha_s))ds\right)^2+2E\left(\int_0^Tf(s,0,0,\alpha_s)ds\right)^2\\
                                                                                  \leq
                                         &\ C\left(E\int_0^T\left(|y_s|^2+|z_s|^2\right)ds+  \sum_{i=1}^mE\int_0^T|f(s,0,0,i)|^2ds\right)
                                        < \infty
                               \end{split}
              \end{equation*}
              which yields that $$\xi +\int_0^Tf(s,y_s,z_s,\alpha_s)ds \in
              L^2({{\cal{G}}_T};R^k).$$
              From Proposition $\ref{Prop:f independent of y and z}$,
              we can define the following contraction mapping under the norm $\|\cdot\|_{\beta}$
              $$
                I((y,z))=(Y,Z):M^2(0,T;R^k \times R^{k\times d})\rightarrow M^2(0,T;R^k \times R^{k\times d}).
              $$
              The proof of contraction property is similar to \cite{Pardoux1990, Karoui1997, Pardoux1999}. For the
              compactness of the paper, the detail is omit here.

              Together with  the  form of BSDE $(\ref{eq:BSDE})$ and Burkholder-Davis-Gundy
              inequality, $Y \in S^2(0,T;R^k)$. Thus,  by the fixed point
              theorem, we know that BSDE $(\ref{eq:BSDE})$ has a unique solution pair.
\end{proof}

\section{BSDEs with Singularly Perturbed Markov Chains} \label{section: BSDE with sungularly perturbed M C}

In this section, after recalling several relevant results of singularly perturbed Markov chains given by Zhang and Yin (\cite{Yin1997}),
 we will consider the asymptotic property of BSDE with a singularly perturbed Markov chain. Following the averaging approach to
 aggregate the states according to their jump rates and replace the actual coefficient with its average with
 respect to the quasi stationary distributions of the singularly perturbed Markov chain, we get the asymptotic probability distribution of the solution
 to the BSDE with an limit averaged Markov chain which has a  much smaller state space than the original one.

\subsection{Relevant results of singularly perturbed Markov chains} \label{subsection:singularly-perturbed Markov chain}

 Focused  on a continuous-time $\varepsilon $-dependent singularly perturbed Markov chain $\alpha^{\varepsilon}=\{\alpha^{\varepsilon}_t;0\leq t\leq T\}$  which have the generator
  $\displaystyle  Q^{\varepsilon}=\frac{1}{\varepsilon}\tilde{Q}+\hat{Q},$ where  $\tilde{Q}$ and $\hat{Q}$ are time-invariant generators, with  $\tilde{Q}=\textrm{diag}(\tilde{Q}^1,\cdots,\tilde{Q}^l)$.  The  state space can be decomposed as
  ${\cal{M}}=\{1,2,\cdots,m\}={\cal{M}}_1\cup \cdots \cup {\cal{M}}_l,$  ${\cal{M}}_k=\{s_{k1},\cdots,s_{k m_k}\}$, and for  $k\in \{1,\cdots,l\}$, $\tilde{Q}^k$ is the weakly irreducible generator\footnote{A generator $Q$ is called weakly irreducible if the
system of equations   $\nu Q=0$ and $  \sum_{i=1}^{m}\nu _i=1$ has a unique nonnegative solution. This nonnegative solution $\nu =(\nu _1,\cdots,\nu _m)$
is called the quasi-stationary distribution of $Q$.}   corresponding to the states in  ${\cal{M}}_k$. The generator $\tilde{Q}$ dictates the fast motion of the Markov
  chain and $\hat{Q}$ governs the slow motion, i.e., the underlying Markov chain fluctuates rapidly in a single group $\mathcal{M}_k$   and jumps less frequently among groups $\mathcal{M}_k$ and $\mathcal{M}_j$ for $k \neq j$.

  As shown in  \cite{Yin1997},  when  the states in $\mathcal{M}_k$ are lumped into a single state, all such states are coupled by $\hat{Q}$.  By defining $\bar{\alpha}^{\varepsilon}_t=k$, when $\alpha^{\varepsilon}_t \in
  {\cal{M}}_k$, we can obtain the aggregated process
$\bar{\alpha}^{\varepsilon}=\{\bar{\alpha}^{\varepsilon}_{t};0\leq
t\leq T\}$ containing $l$ states. The process
  $\bar{\alpha}^{\varepsilon}$ is not necessarily Markovian, but it  converges weakly to a continuous-time
  Markov chain $\bar{\alpha}$.
\begin{prop}(\cite{Yin1997}) \label{prop: converges of aggregated process}
  \noindent (i) $\bar{\alpha}^{\varepsilon}$ converges weakly to
  $\bar{\alpha}$ generated by
  $$
  \bar{Q}=\textrm{diag} (\nu^1,\cdots,\nu^l)\hat{Q}\textrm{diag}
  (\mathbb{I}_{m_1},\cdots,\mathbb{I}_{m_l})
  $$ as $\varepsilon \rightarrow 0$,
  where $\nu^k$ is the quasi-stationary distribution of
  $\tilde{Q}^k$, $k=1,\cdots,l$, and $\mathbb{I}_{k}=(1,\cdots,1)'\in R^k.$

 \noindent  (ii) For any bounded deterministic function $\beta(\cdot)$,
       $$
       E\left(\int_s^T(I_{\{\alpha^{\varepsilon}_t=s_{kj}\}}-\nu_j^k
       I_{\{\bar{\alpha}^{\varepsilon}_t=k\}})\beta(t)dt\right)^2
       =O(\varepsilon), \forall \ k=1,\cdots,l, \forall \
       j=1,\cdots, m_k.
       $$
      Here $I_{A}$ is the indicator function of a set $A$.
\end{prop}

\subsection{Weak  convergence    of  BSDEs with singularly perturbed Markov  chains}

In this subsection, denote $D(0,T;R^k)$ as the Skorohod space of c$\grave{a}$dl$\grave{a}$g
trajectories endowed with the Jakubowski S-topology (\cite{Jakubowski1997}) which is weaker than
the Skorohod topology. As shown in the appendix of \cite{Bahlali2009}, the tightness criteria under this S-topology is
the same as the  ``Meyer-Zheng tightness criteria" used in \cite{Meyer1984}.

Here, we only consider the asymptotic property of the
solution to the following BSDE with a singularly perturbed Markov
chain where the generator $f$ does not depend on
  $Z^{\varepsilon}$,
   \begin{equation} \label{eq:BSDE with perturbed}
                Y^{\varepsilon}_t=\xi+\int_t^Tf(s,Y^{\varepsilon}_s,
                \alpha^{\varepsilon}_s)ds-\int_t^TZ^{\varepsilon}_sdB_s,
  \end{equation}
  For the difficulty to study the general case that the generator $f$ depends on $Z^{\varepsilon}$, we refer interested reader to the explanation in section 6 of   \cite{Pardoux1999}.

Firstly, we make the following assumption:
\begin{assp}\label{assp2:f of BSDE }
            (i) $\xi \in L^2(\mathcal{F}_T^B;R^k).$ (ii) For $f: [0,T]\times R^k \times \mathcal{M} \rightarrow R^k$, there exists a constant $C>0$ such that $\displaystyle \sup_{\substack {0\leq t\leq T\\ 1\leq i\leq m}}|f(t,0,i)|\leq    C$.
\end{assp}
\begin{thm} \label{Thm: weak convergence of Y Z}
                Under Assumption $\ref{assp:f of BSDE}$ and Assumption $\ref{assp2:f
                of BSDE  }$, the sequence of process
                $(Y^{\varepsilon}_t,$               $\int_0^t Z^{\varepsilon}_sdB_s)$ converges  in distribution to the
                process $(Y_t, \int_0^t Z_sd\bar{B}_s)$ as $\varepsilon \rightarrow
                0$, when probability measures on $D(0,T;R^{2k})$ equipped with the Jakubowski S-topology.
                Here $(Y,Z)$ is the solution pair to  the following BSDE with the limit averaged Markov chain
\begin{equation} \label{eq:limit BSDE }
                Y_t=\xi+\int_t^T\bar{f}(s,Y_s,
                \bar{\alpha}_s)ds-\int_t^TZ_sd\bar{B}_s,
\end{equation}
$\bar{B}=\{\bar{B}_t;0\leq t\leq T\}$
with $\bar{B}_0=0$ is a $d$-dimensional Brownian motion, $\bar{\alpha}$ is  defined in
subsection \ref{subsection:singularly-perturbed Markov chain}, and
$\displaystyle \bar{f}(s,y,i)= \sum_{j=1}^{m_i}\nu^i_jf(t,y,s_{ij})$
 for $i \in
\bar{{\cal{M}}} =\{1,\cdots,l\}$.
\end{thm}
\begin{rem}
           It is obvious that the limit BSDE depends on the limit averaged Markov chain $\bar{\alpha}$ with a state space  much smaller than that of the original singularly perturbed Markov chain $\alpha^{\varepsilon}$.
Moreover, as
$\varepsilon\rightarrow 0$, the
$\mathcal{F}^{\alpha^{\varepsilon}}_T$-measurable random variables
sequence $(Y^{\varepsilon}_0)$ converges in distribution to the
                random variable $Y_0$ which is $\mathcal{F}^{\bar{\alpha}}_T$-measurable.
    \end{rem}
For the proof of Theorem \ref{Thm: weak convergence of Y Z}, we follow a classical approach  as in \cite{Pardoux1997, Bahlali2009} to prove the weak convergence of BSDE: after showing the tightness and convergence for $(Y_t^{\varepsilon},\int_0^t Z^{\varepsilon}_sdB_s)$, we identify the limit.


      \noindent {\bf{Step 1: Tightness and convergence for $(Y_t^{\varepsilon},\int_0^t Z^{\varepsilon}_sdB_s)$.}}

%
 \begin{prop} \label{Prop: boundness of Y M}
 Under Assumption $\ref{assp:f of BSDE}$ and Assumption $\ref{assp2:f of BSDE
  }$, BSDE $(\ref{eq:BSDE with perturbed})$ and BSDE $(\ref{eq:limit BSDE })$
   have unique solutions $(Y^{\varepsilon}, Z^{\varepsilon})$ and $  (Y, Z) \in S^2(0,T;R^k)$ $\times M^2(0,T;R^{k\times
  d})$. Moreover, there exists a positive constant $C$ such that $\forall \varepsilon>0$,
  \begin{equation*}
             \begin{split}
                                                E\left(\sup_{0\leq t\leq
                                                T}|Y^{\varepsilon}_t|^2+\int_0^T(Z^{\varepsilon}_t)^2dt\right)
                                                                 &\leq          C,\\
                                               E\left(\sup_{0\leq t\leq
                                               T}|Y_t|^2+\int_0^T(Z_t)^2dt\right)
                                                   & \leq C.
              \end{split}
  \end{equation*}
 \end{prop}
 \begin{proof}
              For BSDE $(\ref{eq:BSDE with perturbed})$, by Theorem \ref{Thm: existence of
              the solution of BSDE}, the existence and uniqueness of solution $(Y^{\varepsilon}, Z^{\varepsilon})$  is obtained for all               $\varepsilon>0$.

              Using It\^o's formula to $|Y^{\varepsilon}_s|^2$ on $[t,T]$,   we get the following from Schwartz's inequality,
              \begin{equation*}
                               \begin{split}
                                               &|Y^{\varepsilon}_t|^2+\int_t^T|Z^{\varepsilon}_s|^2ds\\
                                             = & \ |\xi|^2+2\int_t^TY^{\varepsilon}_s
                                             f(s,Y^{\varepsilon}_s,Z^{\varepsilon}_s,\alpha_s)ds-2\int_t^T
                                                Y^{\varepsilon}_s Z^{\varepsilon}_s dB_s\\
                                          \leq & \
                                          |\xi|^2+2\int_t^T\left((1+\mu^2)|Y^{\varepsilon}_s|^2+|f(s,0,0,\alpha^{\varepsilon}_s)|^2
                                                            \right)ds-2\int_t^T Y^{\varepsilon}_s Z^{\varepsilon}_s
                                                 dB_s
                               \end{split}
              \end{equation*}
               here $\mu$ is the Lipschitz constant of $f$ which is independent of $\varepsilon$. By taking  expectation, we can deduce
               $$
               E\left(|Y^{\varepsilon}_t|^2+\frac{1}{2}\int_t^T|Z^{\varepsilon}_s|^2ds\right) \leq
               |\xi|^2+2\int_t^T((1+\mu^2)|Y^{\varepsilon}_s|^2+|f(s,0,0,\alpha^{\varepsilon}_s)|^2)ds.
               $$ From Gronwall's lemma, we  get
              $$
               \displaystyle
                     E\left(|Y^{\varepsilon}_t|^2+\int_t^T|Z^{\varepsilon}_s|^2ds\right)
                \leq CE\left(|\xi|^2+\int_0^T|f(s,0,0,\alpha^{\varepsilon}_s)|^2ds\right)\leq C,
             $$
             and then the estimation for $(Y^{\varepsilon}, Z^{\varepsilon})$ is obtained from the Burkholder-Davis-Gundy
             inequality.

             From the form  of $\bar{f}$ presented in Theorem \ref{Thm: weak convergence of Y Z}, we know that $\bar{f}$ also satisfies Assumption $\ref{assp:f of BSDE}$ and Assumption $\ref{assp2:f of BSDE
  }$, thus the estimation about $(Y, Z)$ can be obtained similarly.
 \end{proof}

  We set $M_t^{\varepsilon}=\int_0^tZ^{\varepsilon}_sdB_s$  for the convenience. Thus BSDE $(\ref{eq:BSDE with perturbed})$  can be rewritten as
                           \begin{equation} \label{eq: BSDE with perturbed in M form}
                              Y^{\varepsilon}_t=\xi+\int_t^Tf(s,Y^{\varepsilon}_s,
                               \alpha^{\varepsilon}_s)ds-(M_T^{\varepsilon}-M_t^{\varepsilon}).
                            \end{equation}
 \begin{prop} \label{Prop: tightness of Y M}
                                    The sequence of $(Y^{\varepsilon},M^{\varepsilon})$ is tight on the space $D(0,T;$ $R^k) \times
                                    D(0,T;R^k).$
\end{prop}
 \begin{proof}
              Let ${\cal{G}}_t^{\varepsilon}={\cal{F}}^B_t\vee {\cal{F}}_T^{\alpha^{\varepsilon}} \vee {\cal{N}},$
                                     we define the conditional variation
                                     $$
                                          CV(Y^{\varepsilon})=\sup
                                       E\left(
                                       \sum_i|E(Y^{\varepsilon}_{t_{i+1}}-Y^{\varepsilon}_{t_i}|{\cal{G}}^{\varepsilon}_{t_i})
                                       |\right)
                                     $$
               where the supreme is taken over all    partitions of the interval   $[0,T]$.

               From the Proof of Proposition     $\ref{Prop:f independent of y and   z}$, we know that $M^{\varepsilon}$ is a
               ${\cal{G}}_t^{\varepsilon}$-martingale.       It follows that
                  $$
                                          CV(Y^{\varepsilon})
                                          \leq
                                          E\int_0^T|f(s,Y^{\varepsilon}_s,\alpha^{\varepsilon}_s)|ds.
                    $$
                  From   $(ii)$ of Assumption $\ref{assp:f of BSDE}$, $(ii)$ of  Assumption $\ref{assp2:f of BSDE }$, and Proposition                   $\ref{Prop: boundness of Y   M}$, we know
                                     $$
                                        \sup_{\varepsilon}\left(CV(Y^{\varepsilon})+\sup_{0\leq t\leq
                                        T}E|Y^{\varepsilon}_t|+\sup_{0\leq t\leq
                                        T}E|M^{\varepsilon}_t|\right)
                                       < \infty.
                                     $$
                                     Thus the ``Meyer-Zheng tightness criteria" (\cite{Bahlali2009, Meyer1984}) is fully satisfied   and the result is followed.
\end{proof}
Together with the properties of $Y^{\varepsilon}$ obtained above, the following proposition can be seen as an obvious result of Lemma 7.3 in \cite{Zhang2004}.
\begin{prop} \label{prop: convergence of Markov chain}
             Suppose  $g(t,x)$ is a function defined on $[0,T] \times R^m$ satisfying that $g(\cdot,\cdot)$ is Lipschitz continuous
             with $x$ and  $ \forall  x\in R^m$, either $|g(t,x)| \leq K(1+|x|)$ or $|g(t,x)| \leq K$. Denote $\pi_{ij}^{\varepsilon}(t)=\pi_{ij}^{\varepsilon}(t,
\alpha^{\varepsilon}_t)$, with $\pi_{ij}^{\varepsilon}(t,
\alpha)=I_{\{\alpha=s_{ij}\}}-\nu_{j}^iI_{\{\alpha \in M_i\}}$, then
for any $\ k=1,\cdots,l,
       j=1,\cdots, m_k$, $$\sup_{0 <t\leq T}E\left|\int_0^tg(t,Y_s^{\varepsilon}) \pi_{ij}^{\varepsilon}(s,
\alpha^{\varepsilon}_s)ds\right| \rightarrow 0, \textrm{\quad as }
\varepsilon \rightarrow 0.$$
\end{prop}
      \noindent {\bf{Step 2: Identification of the limit.}}
\\[2mm]
       From Proposition $\ref{Prop: tightness of Y  M}$, we know that there exists a subsequence of $(Y^{\varepsilon}, M^{\varepsilon})$, which we still denote by $(Y^{\varepsilon},M^{\varepsilon})$,    and which  converges in distribution on the space $D(0,T;R^k) \times D(0,T;R^k)$ toward a c$\grave{a}$dl$\grave{a}$g process
                                $(\bar{Y},\bar{M})$. Furthermore, there exists a countable subset $D$ of
                                $[0,T]$, such that $(Y^{\varepsilon},M^{\varepsilon})$
                                converges in finite-distribution to
                                $(\bar{Y},\bar{M})$ on $D^c$.
  \begin{prop} \label{Prop: Limit of BSDE with Y,M}
                                            For the limit process $(\bar{Y},\bar{M})$, we have

                                           \noindent (i) For every $t \in [0,T] -
                                                D$,
                                                $$
                                                    \bar{Y}_t=\xi+\int_t^T\bar{f}(s,\bar{Y}_s,\bar{\alpha}_s)ds-(\bar{M}_T-\bar{M}_t).
                                                $$

                                            \noindent (ii) For a $d$-dimensional Brownian motion
                                            $\bar{B}=\{\bar{B}_t;0\leq t\leq T\}$ with $\bar{B}_0=0$,  $\bar{Y}$ is  measurable with  ${\cal{H}}_t={\cal{F}}^{\bar{B}}_t
                                             \vee    {\cal{F}}^{\bar{\alpha}}_{T}$,
                                             then $\bar{M}$ is a
                                             ${\cal{H}}_t$-martingale.
 \end{prop}

\begin{proof}
              From Proposition \ref{prop: convergence of Markov chain}, as $ \varepsilon \rightarrow 0$,
        \begin{equation*}
                      \begin{split}
                                                                & \sup_{0\leq t\leq T}
                                                                  E  \left|  \int_0^t                  f(s,Y^{\varepsilon}_s,s_{ij})
                                  \left(I_{\{\alpha_s^{\varepsilon}=s_{ij}\}}-\nu_j^iI_{\{\alpha_s^{\varepsilon} \in
                                                                 {\cal{M}}_i\}}\right)ds\right|
                                                \rightarrow 0.
   \end{split}
        \end{equation*}
  Since $(Y^{\varepsilon},\bar{\alpha}^{\varepsilon})$ converge weakly to $(\bar{Y},\bar{\alpha})$,
     \begin{equation*}
                      \begin{split}
                      &\int_0^{t}\bar{f}(s,Y^{\varepsilon}_s,\bar{\alpha}^{\varepsilon}_s)ds
                             \textrm{ converges in distribution to}\int_0^{t}\bar{f}(s,\bar{Y}_s,\bar{\alpha}_s)ds
                                         \textrm{ on }  C(0,T;R^k).
                                                    \end{split}
                                               \end{equation*}
                   Thus \begin{equation*}
                                                   \begin{split}
                                                        &\int_0^t    f(s,Y^{\varepsilon}_s,\alpha^{\varepsilon}_s)ds\\
                                                      = &\int_0^t \sum_{i=1}^l \sum_{j=1}^{m_i} f(s,
                                                        Y^{\varepsilon}_s,s_{ij})I_{\{\alpha_s^{\varepsilon}=s_{ij}\}}\\
                                                      = & \int_0^t \sum_{i=1}^l
                                                      \sum_{j=1}^{m_i} f(s, Y^{\varepsilon}_s,s_{ij})\left(I_{\{\alpha_s^{\varepsilon}=s_{ij}\}}-\nu_j^iI_{\{\alpha_s^{\varepsilon} \in
                                                      {\cal{M}}_i\}}\right)ds+
                                                      \int_0^t
                                                      \bar{f}(s,Y^{\varepsilon}_s,\bar{\alpha}^{\varepsilon}_s)ds.
                                                       \end{split}
                                               \end{equation*}
                                                As $\varepsilon \rightarrow 0$, passing to the
                                              limit in the backward
                                              component of the
                                              BSDE $(\ref{eq: BSDE with perturbed in M
                                              form})$, we can derive
                                              assertion (i).

%
Now, we prove assertion  (ii).

 For any $0\leq t_1\leq t_2\leq T$, $\Phi_{t_1}$ is a continuous
 mapping from $C(0,t_1;R^d) \times D(0,t_1;R^k) \times D(0,T;
 \bar{{\cal{M}}})$.
 $\forall  \varepsilon >0$, since $M^{\varepsilon}$ is a martingale with respect to $\mathcal{G}_t^{\varepsilon}=\mathcal{F}_T^{\alpha^{\varepsilon}}\vee \mathcal{F}_t^B$,
 $Y^{\varepsilon}$ and $\bar{\alpha}^{\varepsilon}$ are $\mathcal{G}_t^{\varepsilon}$-adapted, we know
$$
  E\left(\Phi_{t_1}(B,Y^{\varepsilon},  \bar{\alpha}^{\varepsilon})\left(Y^{\varepsilon}_{t_2}-Y^{\varepsilon}_{t_1}+\int_{t_1}^{t_2}
  f(s,Y_s^{\varepsilon},\alpha_s^{\varepsilon})ds\right)\right)=0
 $$
and
 $$
  E\left(\Phi_{t_1}(B,Y^{\varepsilon},
  \bar{\alpha}^{\varepsilon})\int_0^{\delta}(M_{t_2+r}^{\varepsilon}-M_{t_1+r}^{\varepsilon})dr\right)=0,
 $$
 here $B$ is the Brownian motion.

 From the weak convergence of $(Y^{\varepsilon},
 \bar{\alpha}^{\varepsilon})$  to $(\bar{Y},\bar{\alpha})$,
    $\int_0^{t}\bar{f}(s,Y^{\varepsilon}_s,\bar{\alpha}^{\varepsilon}_s)ds$  converges in distribution to
    $\int_0^{t}\bar{f}(s,\bar{Y}_s,\bar{\alpha}_s)ds$
on  $C(0,T;R^k)$. For a $d$-dimensional Brownian motion $\bar{B}=\{\bar{B}_t;0\leq t\leq T\}$ with $\bar{B}_0=0$, from the fact that
  $\bar{B}$ has the same probability distribution
with $B$   and $\displaystyle E(\sup_{0\leq t\leq
T}|M_{t}^{\varepsilon}|^2)\leq
 C$, we obtain
$$
  E\left(\Phi_{t_1}(\bar{B},\bar{Y},\bar{\alpha})\left(\bar{Y}_{t_2}-\bar{Y}_{t_1}+\int_{t_1}^{t_2}\bar{f}(s,\bar{Y}_s,\bar{\alpha}_s)ds\right)\right)=0
 $$
and
 $$
  E\left(\Phi_{t_1}(\bar{B},\bar{Y},
  \bar{\alpha})\int_0^{\delta}(\bar{M}_{t_2+r}-\bar{M}_{t_1+r})dr\right)=0.
 $$
 Dividing the second identity by $\delta$, letting $\delta \rightarrow
 0$, and exploiting the right continuity, we obtain that
 $$
   E\left(\Phi_{t_1}(\bar{B}, \bar{Y},      \bar{\alpha})(\bar{M}_{t_2}-\bar{M}_{t_1})\right)=0.
 $$
 From the freedom choice of $t_1$, $t_2$, and $\Phi_{t_1}$, we deduce
 that $\bar{M}$ is a ${\cal{H}}_t$-martingale.
\end{proof}
 \begin{prop}\label{Prop:identification of M}
              Let $\{(Y_t,Z_t); 0\leq t\leq T\}$ be the unique solution
              of BSDE $(\ref{eq:limit BSDE })$, then  $ \forall  t \in
              [0,T]$,
              $$
                E|Y_t-\bar{Y}_t|^2+E\left([\bar{M}-\int_0^{\cdot}Z_rd\bar{B}_r]_T-[\bar{M}-
                \int_0^{\cdot}Z_rd\bar{B}_r]_t\right)=0.
              $$
 \end{prop}
\begin{proof}
Let $M_t=\int_0^{t}Z_rd\bar{B}_r$,  by the proof of Proposition
$\ref{Prop:f independent of y and z}$, we know that $M_t$ is a
${\cal{F}}_t^{\bar{B}}\vee {\cal{F}}_T^{\bar{\alpha}}$-martingale.

From It\^o's formula and Proposition $\ref{Prop: Limit of BSDE with
Y,M}$, we know that
 \begin{equation*}
       \begin{split}
         &E|Y_t-\bar{Y}_t|^2+E\left([M-\bar{M}]_T-[M-\bar{M}]_t\right)\\
         =\
        &2E\int_t^T\left(\bar{f}(s,Y_s,\bar{\alpha}_s)-\bar{f}(s,\bar{Y}_s,\bar{\alpha}_s)\right)(Y_s-\bar{Y}_s)ds\\
        \leq\ &CE\int_t^T|Y_s-\bar{Y}_s|^2ds.
         \end{split}
  \end{equation*}
  From Gronwall's lemma, we obtain $E|Y_t-\bar{Y}_t|^2=0$, $\forall  t \in
  [0,T]-D$, and the result follows.
  \end{proof}

\vspace{5mm}

\noindent {\bf We come back to finish the Proof of Theorem 3.1}:

   Since $Y$ is continuous, $\bar{Y}$ is c$\grave {a}$dl$\grave{a}$g,    and $D$ is
  countable,    we get $Y_t=\bar{Y}_t$, $P-$a.s., $\forall  t \in
  [0,T]$. Moreover, we can deduce that $M\equiv \bar{M}$. Hence, we
  get the result  that the sequence
  $(Y^{\varepsilon}_t, \int_0^t Z^{\varepsilon}_sdB_s)$ converges in distribution to the process $(Y_t, \int_0^t
  Z_sd\bar{B}_s)$, and the proof of  Theorem \ref{Thm: weak convergence of Y  Z} is
  completed. \qed

\subsection{Examples}

\begin{ex}
          Consider the case that $\tilde{Q}$ is weakly irreducible with the state space $\mathcal{M}=\{1,\cdots,m\}$ and $\nu =(\nu _1,\cdots,\nu _m)$ is the quasi stationary    distribution, then $ \alpha^{\varepsilon}$ can be considered as a fast-varying noise process. As shown in the following, the noise is averaged out with respect to the quasi stationary distribution. In this case, the corresponding BSDE is
          \begin{equation} \label{eq:BSDE a in simple case}
                           Y^{\varepsilon}_t=\xi+\int_t^Tf(s,Y^{\varepsilon}_s,
                           \alpha^{\varepsilon}_s)ds-\int_t^TZ^{\varepsilon}_sdB_s.
          \end{equation}
          Under  Assumption $\ref{assp:f of BSDE}$ and Assumption $\ref{assp2:f of BSDE
  }$, from Theorem $\ref{Thm: weak convergence of Y Z}$,  as $\varepsilon \rightarrow      0$, the sequence of
          process $(Y^{\varepsilon}_t,$
          $\int_0^t Z^{\varepsilon}_sdB_s)$ converges  in distribution to the    process $(Y_t, \int_0^t Z_sd\bar{B}_s)$, where     $(Y,Z)$ is the unique solution to the following   BSDE
                \begin{equation} \label{eq:limit BSDE in simple case }
                                Y_t=\xi+\int_t^T \sum_{i=1}^m \nu_i f(s, Y_s,
                                i)ds-\int_t^TZ_sd\bar{B}_s.
                \end{equation}
           It is noted that the generator of BSDE ($\ref{eq:limit BSDE in simple case }$) depends on  the quasi stationary distribution of the Markov chain.
           Thus we can adopt the distribution of  a $\mathcal{F}_t^{\bar{B}}$-adapted process $Y$, the solution of BSDE \eqref{eq:limit BSDE in simple case }, as the asymptotic distribution for the solution of $\mathcal{F}_t^B \vee
           \mathcal{F}^{\alpha^{\varepsilon}}_{t,T}$-adapted process $Y^{\varepsilon}$.
 \end{ex}
In practical systems, the small parameter  $\varepsilon$ is just a fixed parameter and it separates
  different scales in the sense of order of magnitude in the  generator. It does not need to tend to 0.
  We give a detailed example for interpretation.
\begin{ex}
              Suppose the generator of the continuous-time Markov chain  affected BSDE \eqref{eq:BSDE} is
              $Q=\begin{pmatrix} {-22}&{20}&{2}\\ {41}&{-42}&{1}\\ {1}&{2}&{-3}\end{pmatrix}$,         and the corresponding state space is $\mathcal{M}=\{s_1,s_2,s_3\}$. It is obvious that the transition rate between $s_1$ and $s_2$ is larger than the transition rate
              between $s_3$ and other states, i.e., the jumps between $s_1$ and $s_2$ are more frequent than jumps between $s_3$ and other states. We can rewrite $Q$ as following
              $$\displaystyle Q=\frac{1}{0.05}\tilde{Q}+\hat{Q} =\frac{1}{0.05} \begin{pmatrix}
              {-1}&{1}&{0}\\ {2}&{-2}&{0}\\ {0}&{0}&{0}\end{pmatrix}+\begin{pmatrix} {-2}&{0}&{2}\\ {1}&{-2}&{1}\\ {1}&{2}&{-3}\end{pmatrix}$$
              It is noted that we choose suitable $\varepsilon$ to guarantee that $\tilde{Q}$
              and $\hat{Q}$ to be the generator with the same order of magnitude.

              Now, we introduce the continuous-time $\varepsilon $-dependent singularly perturbed Markov chain $\alpha^{\varepsilon}=\{\alpha^{\varepsilon}_t;0\leq t\leq T\}$ which
               have     the generator
                $\displaystyle  Q^{\varepsilon}=\frac{1}{\varepsilon}\tilde{Q}+\hat{Q}=\frac{1}{\varepsilon}
                \begin{pmatrix}             {-1}&{1}&{0}\\ {2}&{-2}&{0}\\ {0}&{0}&{0}\end{pmatrix}+\begin{pmatrix} {-2}&{0}&{2}\\ {1}&{-2}&{1}\\ {1}&{2}&{-3}
                \end{pmatrix},$  and define the aggregated process
              $$\bar{\alpha}^{\varepsilon}=\{\bar{\alpha}^{\varepsilon}_{t};0\leq t\leq T\}=\left \{ \begin{array}{ll}
                                                                                                                    1, \  \alpha^{\varepsilon}_t \in \{s_1, s_2\} \\
                                                                                                                    2, \  \alpha^{\varepsilon}_t \in \{s_3\}
                                                                                                     \end{array}
                                                                                            \right
                                                                                            .$$
               Proposition \ref{prop: converges of aggregated process} yields that $\bar{\alpha}^{\varepsilon}$ converges in distribution to a continuous-time Markov chain $\bar{\alpha}$
               generated by $\displaystyle \bar{Q}=\begin{pmatrix} -\frac{5}{3} & \frac{5}{3} \\ {3} & -{3} \end{pmatrix}.$ By Theorem \ref{Thm: weak convergence of Y Z}, we can
               adopt the        probability distribution of the     solution to the following BSDE \begin{equation*}
                                                            Y_t=\xi+\int_t^T\bar{f}(s,Y_s,
                                                            \bar{\alpha}_s)ds-\int_t^TZ_sd\bar{B}_s
                                          \end{equation*}
               as an asymptotic   probability distribution of the solution to the original BSDE.
               Here $\displaystyle \bar{f}(t,y,1)=\frac{2}{3}f(t,y,s_1)+\frac{1}{3}f(t,y,s_2)$ and $ \bar{f}(t,y,2)=f(t,y,s_3).$

               Since the limit averaged Markov chain has two states and the original one has three states,
               we have reduced the complexity of the model. This advantage will be more clear
                when the state space of the original Markov chain is sufficiently larger.
\end{ex}

\section{Homogenization of One System of PDEs} \label{section: homogenization
of SPDEs}

As an application of our results in previous section, we show the homogenization of a sequence of semi-linear backward
PDE with a singularly perturbed Markov chain. In this section, after showing the relation between BSDEs with Markov chain  and one system of semi-linear PDE with Markov chain, we derive the homogenization property of backward
PDE with a singularly perturbed Markov chain based on the weak
convergence of the associated BSDE.

Here, we give some notations as follows: $C^k(R^p;R^q)$ is the space
of functions of class $C^k$ from $R^p$ to $R^q$,
$C^k_{l,b}(R^p;R^q)$ is the space of functions of class $C^k$ whose
partial derivatives of order less than or equal to $k$ are bounded,
and $C^k_p(R^p;R^q)$ is the space of functions of class $C^k$ which,
together with all their partial derivatives of order less than or
equal to $k$, grow at most like a polynomial function of the
variable $x$ at infinity.


\subsection{Relation between BSDEs  with Markov chains and  semi-linear
PDEs systems with Markov chains}

For $t \in [0,T]$, consider the following semi-linear backward PDE with a Markov chain:
 \begin{equation}\label{Eq:SPDEs }
          u(t,x)=h(x)+\int_t^T\left(\mathcal{L}u(r,x)+f(r,x,u(r,x),(\nabla u
\sigma)(r, x),\alpha_r)\right)dr,
  \end{equation}
  here $u: [0,T]\times R^m\rightarrow R^k$, and $\displaystyle
\mathcal{L}u=\left(Lu_1,\cdots,Lu_k \right)'$, with $\allowdisplaybreaks \displaystyle
L=\frac 1 2 \sum_{i,j=1}^m(\sigma \sigma')_{ij}$ $\displaystyle (t,x)
\frac{\partial^2}{\partial x_i \partial x_j}+ \sum_{i=1}^m
b_i(t,x)\frac{\partial}{\partial x_i}.$

Firstly, we make the following assumption:
\begin{assp} \label{assp:coefficient of FBSDE}
 $b \in C^3_{l,b}(R^m;R^m)$, $\sigma \in C^3_{l,b}(R^m;R^{m\times
 d})$, $h \in C_p^3(R^m;R^k)$.
 For $f: [0,T]\times R^m \times R^k \times R^{k \times d} \times \mathcal{M} \rightarrow
 R^k$, $ \forall  s \in [0,T]$, $i \in \mathcal{M}$,  $(x,y,z)\rightarrow
 f(s,x,y,z,i)$ is of class $C^3$.

 Moreover, $f(s,\cdot,0,0,i) \in
 C^3_{p}(R^m;R^k)$, and the first order partial derivatives in $y$
 and $z$ are bounded on $[0,T]\times R^m \times R^k \times R^{k \times
 d} \times \mathcal{M}$, as well as their derivatives of order one and two with respect
 to $x$, $y$, $z$.
 \end{assp}

  \begin{defi}
      A classical solution of PDE \eqref{Eq:SPDEs }
      is a $R^k$-valued stochastic process $\{u(t,x);$
      $0\leq t\leq T,x \in
      R^m\}$ which is in $C^{0,2}([0,T]\times R^m;R^k)$ and satisfies that  $u(t,x)$
      is $\mathcal{F}_{t,T}$-measurable.
\end{defi}

  $ \forall t \in [0,T]$, $x \in R^m$, we introduce the
following FBSDE  with a Markov chain on $[t, T]$:
\begin{equation}\label{Eq:SDE of FBSDE}
          X_s^{t,x}=x+\int_t^sb(X_r^{t,x})dr+\int_t^s\sigma(X_r^{t,x})dB_r,
  \end{equation}
   \begin{equation}\label{Eq:BSDE of FBSDE}
                Y_s^{t,x}=h(X_T^{t,x})+\int_s^Tf(r,X_r^{t,x},Y_r^{t,x},Z_r^{t,x},\alpha_r)dr-\int_s^TZ_r^{t,x}dB_r.
 \end{equation}
The aim of this subsection  is to show that, under above
assumptions, the FBSDE (\ref{Eq:SDE of FBSDE})-(\ref{Eq:BSDE of
FBSDE}) provides both a probabilistic representation  and the unique
classical solution for PDE (\ref{Eq:SPDEs }).

For  SDE (\ref{Eq:SDE of FBSDE}),    it is well known
that under Assumption \ref{assp:coefficient of FBSDE},
 it has a unique solution $\{ X_s^{t,x};t\leq s\leq
 T\}$ which has a version that is a.s. of class $C^2$ in $x$,
 the function and its derivatives are a.s. jointly continuous in
 $(t,s,x)$. Moreover,
 $$\sup_{t\leq s\leq T}\left(|X_s^{t,x}|+|\nabla X_s^{t,x}|+|D^2X_s^{t,x}|\right) \in \bigcap_{p\geq 1}L^p(R), \forall  (t,x) \in [0,T] \times R^m $$
 where $\nabla X_s^{t,x}$, $D^2X_s^{t,x}$  denote respectively the matrix of
 first order and second order derivatives of $X_s^{t,x}$ with
 respect to $x$.

 For  BSDE (\ref{Eq:BSDE of FBSDE}), denote $\tilde{f}(s,y,z,i)=f(s,X^{t,x}_s,y,z,i)$, $\forall  i \in \mathcal{M}$,
  we know that $\tilde{f}$ satisfies Assumption \ref{assp:f of BSDE} since $f$ satisfies  Assumption \ref{assp:coefficient of
 FBSDE}. So there exists a unique  solution pair  $\{(Y_s^{t,x},Z_s^{t,x});t\leq s\leq
 T\}$ to BSDE (\ref{Eq:BSDE of FBSDE}).

 Define $X_s^{t,x}=X_{s\vee t}^{t,x}$, $Y_s^{t,x}=Y_{s\vee
t}^{t,x}$, and $Z_s^{t,x}=0$, for $s \leq t$. Then
$(X,Y,Z)=(X_s^{t,x},Y_s^{t,x},Z_s^{t,x})$ is defined on $(s,t) \in
[0,T]^2$.

\begin{thm} \label{Thm: probalistic of soulution of PDE}
      Under Assumption $\ref{assp:coefficient of FBSDE}$, let $\{u(t,x);0\leq t\leq T,x\in R^m\}$
      be
      a classical solution of PDE $(\ref{Eq:SPDEs })$. Suppose that there exists a constant $C$ such
      that,
      \begin{equation} \label{in Thm: probalistic of soulution of PDE}
      |u(t,x)|+|\partial_xu(t,x)\sigma(t,x)|\leq C(1+|x|), \quad \forall  (t,x) \in [0,T]\times R^m,
      \end{equation}  then $
      (Y_s^{t,x}=u(t,X_s^{t,x}), Z_s^{t,x}=\partial_xu(t,X_s^{t,x})\sigma(t,X_s^{t,x});t\leq s\leq T)$ is the unique solution of
      BSDE $(\ref{Eq:BSDE of FBSDE})$. Here $(X_s^{t,x}; t\leq s\leq
      T)$ is the solution to SDE $(\ref{Eq:SDE of FBSDE})$.
\end{thm}

\noindent {{\bf Proof:}} $\forall  t \leq s \leq T$, let $s=t_0<
t_1< t_2< \cdots < t_n=T$, with It\^o's formula and PDE
$(\ref{Eq:SPDEs })$, we get
 \begin{equation*}
                \begin{split}
                & Y_s^{t,x}-h(X_T^{t,x})\\
                               =\ & u(s,x)-u(T,X_T^{t,x})\\
                =\
                &
                \sum_{i=0}^{n-1}\left(u(t_i,X_{t_i}^{t,x})-u(t_{i+1},X_{t_{i+1}}^{t,x})\right)\\
                = \ &   \sum_{i=0}^{n-1}\left(u(t_i,X_{t_i}^{t,x})-u(t_{i},X_{t_{i+1}}^{t,x})\right)+\sum_{i=0}^{n-1}\left(
                u(t_{i},X_{t_{i+1}}^{t,x})-u(t_{i+1},X_{t_{i+1}}^{t,x})\right) \end{split}
 \end{equation*}
  \begin{equation*}
                \begin{split} = \ & \sum_{i=0}^{n-1} \bigg( -\int_{t_i}^{t_{i+1}} \left(\mathcal{L}u(t_i,X_{s}^{t,x})ds-(\nabla u
                  \sigma)(t_i, X_s^{t,x})dB_s\right) \\
                  &+ \int_{t_i}^{t_{i+1}}\left(\mathcal{L}u(s,X_{t_{i+1}}^{t,x})+f(s,X_{t_{i+1}}^{t,x},u(r,X_{t_{i+1}}^{t,x}),(\nabla u
                  \sigma)(s, X_{t_{i+1}}^{t,x}),\alpha_s)\right)ds\bigg).
                \end{split}
 \end{equation*}
\eqref{in Thm: probalistic of soulution of PDE} yields that $$ E\left(\sup_{t\leq s\leq
T}|u(t,X_s^{t,x})|^2+\int_t^T|\partial_xu \sigma(s,X_s^{t,x})|^2ds\right)
< \infty,$$ and the adaptability is obvious.
 The result is followed as
$\displaystyle \triangle=\sup_{0\leq i\leq n-1} |t_{i+1}-t_i|
\rightarrow 0$. \qed

 Now we deduce the converse side of Theorem \ref{Thm:
probalistic of soulution of PDE}.

\begin{thm} \label{Thm: converse of probalistic of soulution of PDE}
    Assume that        for some
             $p>2$, $ E|\xi|^p+E\int_0^T|\tilde{f}(t,0,0,\alpha_t)|^pdt <
             \infty,$   let $b, \sigma, f, h, \alpha $ satisfy
      Assumption $\ref{assp:coefficient of FBSDE}$,
      then the process $\{u(t,x)=Y_t^{t,x}; 0 \leq t\leq T, x \in R^m\}$ is the unique classical solution to PDE $(\ref{Eq:SPDEs })$.
\end{thm}
As preliminaries for the proof, we  give  two propositions about the
regularity of the solution of BSDE (\ref{Eq:BSDE of FBSDE}) whose
proofs are put in the Appendix.

\begin{prop}\label{Prop:continuous of the solution of BSDE of FBSDE}
Under the assumption of Theorem $\ref{Thm: converse of probalistic
of soulution of PDE}$,
          $\{Y_s^{t,x};(s,t)\in [0,T]^2, x \in R^m\}$ has a version
          whose trajectories belong to $C^{0,0,2}([0,T]^2\times
          R^m)$. Hence $\forall  t \in [0,T]$, $x\rightarrow
          Y_t^{t,x}$ is of class $C^2$ a.s..
\end{prop}

\begin{prop}\label{Prop: representation of Z of BSDE of FBSDE}
           Under the assumption of Theorem $\ref{Thm: converse of probalistic of
soulution of PDE}$, $\{Z_s^{t,x};(s,t)\in [0,T]^2, x \in R^m\}$ has
an a.s. continuous
           version which is given by $Z_s^{t,x}=\nabla Y_s^{t,x}(\nabla X_s^{t,x})^{-1}\sigma
           (X_s^{t,x})$. In particular, $Z_t^{t,x}=\nabla Y_t^{t,x}\sigma
           (x)$. Here $\displaystyle \Big(\nabla Y_s^{t,x}=\frac {\partial Y_s^{t,x}}{\partial x},$ $\displaystyle \nabla Z_s^{t,x}=\frac {\partial Z_s^{t,x}}{\partial x} \Big)$
           is the unique solution of
           \begin{equation*}
                      \begin{split}
                             \nabla Y_s^{t,x}=&h'(X_T^{t,x})\nabla
                             X_T^{t,x} +\int_s^T \big(f'_x(r,X_r^{t,x},Y_r^{t,x},Z_r^{t,x},\alpha_r)\nabla
                             X_r^{t,x}+f'_y(r,X_r^{t,x},\\
                             &Y_r^{t,x},Z_r^{t,x},\alpha_r)\nabla
                             Y_r^{t,x}+f'_z(r,X_r^{t,x},Y_r^{t,x},Z_r^{t,x},\alpha_r)\nabla
                             Z_r^{t,x}\big)dr-\int_s^T Z_r^{t,x} dB_r.
                      \end{split}
           \end{equation*}
 \end{prop}

 \noindent {{\bf Proof of Theorem \ref{Thm: converse of probalistic of soulution of PDE}:}}
Let $t=t_0<t_1<\cdots <t_n=T$, we have
 \begin{equation*}
                      \begin{split}
                      &h(x)-u(t,x)\\
                      = \ &u(T,x)-u(t,x) \\
                                           =&\sum_{i=0}^{n-1}\left(u(t_{i+1},x)-u(t_i,x)\right)\\
                     =&\sum_{i=0}^{n-1}\left(u(t_{i+1},x)-u(t_{i+1},X_{t_{i+1}}^{t_i,x})+u(t_{i+1},X_{t_{i+1}}^{t_i,x})-u(t_i,x)\right).
                      \end{split}
    \end{equation*}
    Since $u(t_{i+1},X_{t_{i+1}}^{t_i,x})=Y_{t_{i+1}}^{t_{i+1},X_{t_{i+1}}^{t_i,x}}=Y_{t_{i+1}}^{t_i,x}$,
   we obtain the following from BSDE (\ref{Eq:BSDE of FBSDE})
 \begin{equation*}
                      \begin{split}
                     &u(t_{i+1},X_{t_{i+1}}^{t_i,x})-u(t_i,x)\\
                    = \ &Y_{t_{i+1}}^{t_i,x}-Y_{t_{i}}^{t_i,x}\\
                    = \ & -\int_{t_i}^{t_{i+1}} f(r,X_r^{t_i,x},Y_r^{t_i,x},Z_r^{t_i,x},\alpha_r)dr+
                                   \int_{t_i}^{t_{i+1}}
                                   Z_r^{t,x}dB_r.
                      \end{split}
    \end{equation*}
   It is known that $u(t,\cdot) \in
   C^2(R^m)$ from Proposition \ref{Prop:continuous of the solution of BSDE of
   FBSDE}. Then, with It\^o's formula, we  get
 \begin{align*}
                      \begin{split}
                      &h(x)-u(t,x)\\
                      = \ & \sum_{i=0}^{n-1}\bigg(\int_{t_i}^{t_{i+1}}\mathcal{L}u(t_{i+1},X_{r}^{t_i,x})dr-
                   \int_{t_i}^{t_{i+1}} (\nabla u
                                    \sigma)(t_{i+1},X_{r}^{t_i,x})dB_r\\
                   &                                   -\int_{t_i}^{t_{i+1}} f(r,X_r^{t_i,x},Y_r^{t_i,x},Z_r^{t_i,x},\alpha_r)dr+
                                   \int_{t_i}^{t_{i+1}} Z_r^{t,x}dB_r
                                   \bigg)\\
                             =&-\sum_{i=0}^{n-1}\int_{t_i}^{t_{i+1}}\left(\mathcal{L}u(t_{i+1},X_{r}^{t_i,x})+f(r,X_r^{t_i,x},
                             Y_r^{t_i,x},Z_r^{t_i,x},\alpha_r)\right)dr\\
                                    &
                                   + \sum_{i=0}^{n-1}\int_{t_i}^{t_{i+1}}\left(Z_r^{t_i,x}-(\nabla u
                                    \sigma)(t_{i+1},X_{r}^{t_i,x})\right)dB_r.
                       \end{split}
    \end{align*}
    From Proposition \ref{Prop:continuous of the solution of BSDE of
    FBSDE} and Proposition \ref{Prop: representation of Z of BSDE of
    FBSDE},  letting $\displaystyle \triangle =\sup_{0\leq i \leq n-1} |t_{i+1}-t_i| \rightarrow 0
    $, we have
     \begin{equation*}
          u(t,x)=h(x)+\int_t^T\left(\mathcal{L}u(r,x)+f(r,x,u(r,x),(\nabla u
\sigma)(r, x),\alpha_r)\right)dr,
  \end{equation*}
  here $u \in C^{0,2}([0,T]\times R^m; R^k)$. The uniqueness  property is
  followed from Theorem \ref{Thm: probalistic of soulution of PDE}
  and the uniqueness of the solution of  BSDE (\ref{Eq:BSDE of FBSDE}). \qed
  \subsection{Homogenization of PDEs system with a singularly perturbed  Markov chain}

Now we can give the application of our theoretical result in previous section (Theorem
3.1): homogenization of PDEs system with a singularly perturbed  Markov chain.

Consider the following sequence of semi-linear backward PDE with a
singularly perturbed Markov chain, indexed by $\varepsilon >0$, for $t \in [0, T],  x \in R^m,$
 \begin{equation}\label{Eq:SPDEs with parameter}
          u^{\varepsilon}(t,x)=h(x)+\int_t^T\left(\mathcal{L}u^{\varepsilon}(r,x)+f(r,x,u^{\varepsilon}(r,x),\alpha_r^{\varepsilon})\right)dr,\quad
   \end{equation}
Here  $\alpha^{\varepsilon}$ is the singularly perturbed Markov chain
which is stated in subsection \ref{subsection:singularly-perturbed Markov
chain}. We have the following homogenization result.
  \begin{thm} \label{Thm: homogenization of BSDE}
      Under Assumption $\ref{assp2:f of BSDE
      }$ and Assumption $\ref{assp:coefficient of FBSDE}$, PDE $(\ref{Eq:SPDEs with
      parameter})$ has a classical solution $\{u^{\varepsilon}(t,x);0\leq t\leq T,x\in R^m\}$.  As $\varepsilon \rightarrow
                0$,      the sequence of $u^{\varepsilon}$ converges in distribution
      to a                process $u$, where $u(t,x)$ is the classical solution of the following
               PDE with the limit averaged Markov chain $\bar{\alpha}$
                 \begin{equation}\label{Eq:limt of SPDEs }
                          u(t,x)=h(x)+\int_t^T\left(\mathcal{L}u(r,x)+ \bar{f}(r,x,u(r,x),\bar{\alpha}_r)\right)dr,\quad
                           0\leq t\leq T.
                \end{equation}
                Here  $\bar{f}$ is the average of $f$ defined as
                $ \displaystyle
                \bar{f}(t,x,u,i)=\sum_{j=1}^{m_i}\nu_j^if(t,x,u,s_{ij})$, for $i \in\bar{ \mathcal{M}}=\{1,\cdots,l\}$.
\end{thm}
\begin{proof} From Theorem \ref{Thm: converse of
probalistic of soulution of PDE}, we know that
$\{u^{\varepsilon}(t,x)=Y_t^{\varepsilon,t,x}; 0 \leq t\leq T, x \in
R^m\}$ is the unique classical solution of PDE (\ref{Eq:SPDEs with
parameter}) where $\{Y_s^{\varepsilon,t,x};t \leq s \leq T\}$
satisfies
  \begin{equation}\label{Eq:BSDE of FBSDE with parameter}
                Y_s^{\varepsilon,t,x}=h(X_T^{t,x})+\int_s^Tf(r,X_r^{t,x},Y_r^{\varepsilon,t,x},\alpha_r^{\varepsilon})dr
                -\int_s^TZ_r^{\varepsilon,t,x}dB_r,
 \end{equation}
 and $\{X_s^{t,x}; t\leq s \leq T\}$ satisfies SDE (\ref{Eq:SDE of
 FBSDE}). $\forall  (t,x) \in [0,T] \times R^m$,  from Theorem \ref{Thm: weak
convergence of Y
 Z}, we obtain that $Y_t^{\varepsilon,t,x}$ converges  in distribution to
 $Y_t^{t,x}$ as $\varepsilon \rightarrow 0$ where $\{Y_s^{t,x};t
\leq s \leq T\}$ satisfies
  \begin{equation}\label{Eq:limit of BSDE of FBSDE with parameter}
                Y_s^{t,x}=h(X_T^{t,x})+\int_s^T\bar{f}(r,X_r^{t,x},Y_r^{t,x},\bar{\alpha}_r)dr-\int_s^TZ_r^{t,x}d\bar{B}_r,
   \end{equation}
 Again from Theorem \ref{Thm: converse of probalistic of soulution of
 PDE}, we know that $u(t,x)=Y_t^{t,x}$ is the unique classical solution
to  PDE (\ref{Eq:limt of SPDEs }), and the results are followed.
 \end{proof}
\section{Conclusion}
In this paper, stemmed from the adjoint equation for deriving the optimal control of stochastic LQ control problem with Markovian jumps, we study the solvability of one kind of BSDE with the generator depending on a Markov switching. Then, we consider the case that the Markov chain has a large state space. To reduce the complexity, we adopt a hierarchical approach and study the asymptotic property of BSDE with a singularly perturbed Markov chain. Also, as an application, we present the homogenization property of one system of PDE with   a singularly perturbed Markov chain.

It is noted that in this paper, we only give the homogenization result of PDEs system with Markov chains when
   there exists classical solution under smooth assumptions. In the successive work, we will study
   the Sobolev space weak solution for the related PDEs system and homogenization problem by virtue of BSDEs with
   Markov chain. Some applications of this kind of BSDEs in optimal control and mathematics financial problems would
   also be interesting
   to investigate in our future research.


\section*{Acknowledgements}
It is our great pleasure to express the thankfulness to Professor
Qing Zhang in University of Georgia for many useful discussions and
suggestions.

 \appendix

\section{Proof of Proposition \ref{Prop:continuous of the solution of
BSDE of FBSDE} and Proposition \ref{Prop: representation of Z of
BSDE of FBSDE}}

The proof of Proposition \ref{Prop:continuous of the solution of
BSDE of FBSDE} and Proposition \ref{Prop: representation of Z of
BSDE of FBSDE} follow a classic approach as shown in \cite{Pardoux1992, Pardoux1994}. Here, we will give a sketch of the proof.
   Firstly, we present a higher order moment estimation
to the solution of BSDE \eqref{eq:BSDE}.

\begin{cor} \label{coro:higher order estimate of solution of BSDE}
             Assume that        for some
             $p>2$, $ E|\xi|^p+E\int_0^T|f(t,0,$ $0,\alpha_t)|^pdt <
             \infty,$  under Assumption $\ref{assp:f of BSDE}$,  we have the
             following estimation for BSDE $(\ref{eq:BSDE})$,
             \\[-2mm]
             $$\displaystyle
                     E\left(\sup_{0\leq s\leq
                     t}|Y_s|^p+(\int_0^tZ_s^2ds)^{\frac p 2}\right)
               < \infty, \quad \forall  0 \leq t\leq T.
             $$
\end{cor}

\begin{proof}Applying It\^o's formula to $|Y_t|^p$ from
$t$ to $T$, we can get
\begin{equation*}
       \begin{split}
               & |Y_t|^p+\frac {p(p-1)} {2}
               \int_t^T|Y_s|^{p-2}|Z_s|^2ds\\
               = & \ |\xi|^p+p\int_t^T|Y_s|^{p-2}
               Y_s
               f(s,Y_s,Z_s,\alpha_s)ds-p\int_t^T|Y_s|^{p-2}Y_sZ_sdB_s.
       \end{split}
\end{equation*}
By the same technique as that in Lemma 2.1 of  Pardoux and Peng
\cite{Pardoux1992}, we obtain that
\begin{equation*}
       \begin{split}
               & E|Y_t|^p+\frac {p(p-1)} {2}
               E\int_t^T|Y_s|^{p-2}|Z_s|^2ds\\
               \leq & \ E|\xi|^p+p E \int_t^T|Y_s|^{p-2}
               Y_s
               f(s,Y_s,Z_s,\alpha_s)ds.
       \end{split}
\end{equation*}
From Assumption \ref{assp:f of BSDE}, using H$\ddot{\textrm{o}}$lder
and Young's inequalities, there exist $K >0$ and $C$ such that
\begin{equation*}
       \begin{split}
               & E|Y_t|^p+K
               E\int_t^T|Y_s|^{p-2}|Z_s|^2ds\\
               \leq & \ E|\xi|^p+C E \int_t^T\left(|Y_s|^{p}+
               |f(s,0,0,\alpha_s)|^p\right)ds.
       \end{split}
\end{equation*}
It follows from Gronwall's lemma that
$$\sup_{0\leq t \leq T}E|Y_t|^p+E \int_0^T |Y_t|^{p-2}|Z_t|^2dt < \infty.$$
Since
\begin{equation*}
       \begin{split}
                |Y_t|^p
               \leq  \ |\xi|^p+p\int_t^T|Y_s|^{p-2}
               Y_s
               f(s,Y_s,Z_s,\alpha_s)ds-p\int_t^T|Y_s|^{p-2}Y_sZ_sdB_s,
       \end{split}
\end{equation*}
Burkholder-Davis-Gundy inequality yields that $E(\sup_{0\leq t \leq T}|Y_t|^p) < \infty$.

Now we prove $E(\int_0^tZ_s^2ds)^{\frac p 2}
               < \infty$. Since $$\int_0^tZ_s dB_s=Y_t-Y_0+\int_0^tf(s,Y_s,Z_s,\alpha_s)ds,$$
$$\sup_{0\leq t \leq T}|\int_0^tZ_s dB_s| \leq 2 \sup_{0\leq t \leq T} |Y_t|+\int_0^T|f(s,Y_s,Z_s,\alpha_s)|ds,$$
the result is  followed from Assumption \ref{assp:f of BSDE} and
Burkholder-Davis-Gundy inequality.
\end{proof}

\begin{lem}(Lemma 2.7 in
\cite{Pardoux1992}) For any $p>2$, there exists a  constant $c_p$ such that
for any $t,t' \in [0,T]$, $x,x' \in R^m$, $i \in \{1,\cdots,d\}$,
$h,h' \in R \backslash \{0\}$,
$$E ( \sup_{0 \leq s \leq T} |X_s^{t,x}|^p) \leq c_p(1+|x|^p),$$
$$ E ( \sup_{0 \leq s \leq T} |X_s^{t,x}-X_s^{t',x'}|^p) \leq c_p(1+|x|^p) (|x-x'|^p+|t-t'|^{\frac{p}{2}}),$$
$$E ( \sup_{0 \leq s \leq T} |\triangle_h^i X_s^{t,x}|^p) \leq c_p,$$
$$E ( \sup_{0 \leq s \leq T} |\triangle_h^i X_s^{t,x}-\triangle_{h'}^i X_s^{t',x'}|^p) \leq c_p
(|x-x'|^p+|h-h'|^p+|t-t'|^{\frac{p}{2}}).$$ Here $\displaystyle
\triangle_h^ig(x)=\frac{g(x+he_i)-g(x)}{h}$, $1\leq i \leq d$, where
$e_i$ denotes the $i$th vector of an arbitrary orthonormal basis of
$R^m$.
\end{lem}

\noindent {{\bf Proof of Proposition \ref{Prop:continuous of the solution of
BSDE of FBSDE}:}}  Since $$E ( \sup_{0 \leq s \leq T} |X_s^{t,x}|^p)
\leq c_p(1+|x|^p),$$ from the proof of Corollary \ref{coro:higher
order estimate of solution of BSDE}, $\forall  p>2$, there exist
$C_p$ and $q$ such that  $$ E\left(\sup_{0\leq s\leq
                     t}|Y_s^{t,x}|^p+(\int_0^t|Z_s^{t,x}|^2ds)^{\frac p
                     2}\right)
               \leq C_p(1+|x|^q).$$
               Note that for $t\vee t' \leq s\leq T$
\begin{equation*}
     \begin{split}
          &Y_s^{t,x}-Y_s^{t',x'}\\
          =& \ \bigg (\int_0^1h'(X_T^{t,x}+\lambda
          (X_T^{t,x}-X_T^{t',x'}))d\lambda \bigg)(X_T^{t,x}-X_T^{t',x'})\\
          &+\int_s^T \int_0^1 \bigg (f'_x(\Xi_{r,\lambda}^{t,x,t',x'},\alpha_r)(X_r^{t,x}-X_r^{t',x'})+f'_y(\Xi_{r,\lambda}^{t,x,t',x'},\alpha_r)(Y_r^{t,x}-Y_r^{t',x'})\\
          &+f'_z(\Xi_{r,\lambda}^{t,x,t',x'},\alpha_r)(Z_r^{t,x}-Z_r^{t',x'})\bigg)d\lambda dr-\int_s^T(Z_r^{t,x}-Z_r^{t',x'})dB_r,
     \end{split}
\end{equation*}
where $\displaystyle
     \Xi_{r,\lambda}^{t,x,t',x'}=(r,X_r^{t',x'}+\lambda
          (X_r^{t,x}-X_r^{t',x'}),Y_r^{t',x'}+\lambda
          (Y_r^{t,x}-Y_r^{t',x'}),Z_r^{t',x'}+\lambda
          (Z_r^{t,x}-Z_r^{t',x'})).$
Since $$ E ( \sup_{0 \leq s \leq T} |X_s^{t,x}-X_s^{t',x'}|^p) \leq
c_p(1+|x|^p) (|x-x'|^p+|t-t'|^{\frac{p}{2}}),$$ combing with the
proof of Corollary \ref{coro:higher order estimate of solution of
BSDE}, we can deduce that $\forall p\geq 2$, there exist $C_p$ and
$q$ such that
\begin{align*}
 &     E \left( \sup_{0 \leq s \leq T} |Y_s^{t,x}-Y_s^{t',x'}|^p +\left(\int_t^T |Z_s^{t,x}-Z_s^{t',x'}|^2ds\right)^{\frac{p}{2}}\right) \\
 \leq  \ \ & C_p(1+|x|^q) (|x-x'|^p+|t-t'|^{\frac{p}{2}}).
\end{align*}
Then using Kolmogorov's lemma, we know that $\{Y_s^{t,x};(s,t)\in
[0,T]^2, x\in R^m\}$ has an a.s. continuous version.

Next, we have
\begin{equation*}
     \begin{split}
     \triangle_h^iY_s^{t,x}=& \int_0^1h'(X_T^{t,x}+\lambda h
     \triangle_h^iX_T^{t,x})\triangle_h^iX_T^{t,x}d \lambda +
     \int_s^T \int_0^1 \big(f'_x(\Theta_{r,\lambda}^{t,x,h},\alpha_r)\triangle_h^iX_r^{t,x}\\
     &+
     f'_y(\Theta_{r,\lambda}^{t,x,h},\alpha_r)\triangle_h^iY_r^{t,x}
     +f'_z(\Theta_{r,\lambda}^{t,x,h},\alpha_r)\triangle_h^iZ_r^{t,x}\big)d\lambda
     dr- \int_s^T \triangle_h^iZ_r^{t,x}dB_r
     \end{split}
\end{equation*}
where $\Theta_{r,\lambda}^{t,x,h}=(r,X_r^{t,x}+\lambda h
     \triangle_h^iX_r^{t,x}, Y_r^{t,x}+\lambda h
     \triangle_h^iY_r^{t,x}, Z_r^{t,x}+\lambda h
     \triangle_h^iZ_r^{t,x})$.

     Since for each $p \geq 2$, there exists $c_p$ such that
     $$E ( \sup_{0 \leq s \leq T} |\triangle_h^i X_s^{t,x}|^p) \leq c_p.$$
     We can have the following estimation
     $$\displaystyle E\left(\sup_{t\leq s \leq T}|\triangle_h^i Y_s^{t,x}|^p+(\int_t^T|\triangle_h^i Z_s^{t,x}|ds)^{\frac{p}{2}}\right)
       \leq c_p(1+|x|^q+|h|^q).     $$
       Then we consider
       \begin{equation*}
     \begin{split}
   & \triangle_h^i Y_s^{t,x}-\triangle_{h'}^i Y_s^{t',x'}\\
   = & \int_0^1h'(X_T^{t,x}+\lambda h
     \triangle_h^iX_T^{t,x})\triangle_h^iX_T^{t,x}d \lambda -\int_0^1h'(X_T^{t',x'}+\lambda h
     \triangle_{h'}^iX_T^{t',x'})\triangle_{h'}^iX_T^{t',x'}d
     \lambda\\
     &+  \int_s^T \int_0^1 (f'_x(\Theta_{r,\lambda}^{t,x,h},\alpha_r)\triangle_h^iX_r^{t,x}-
      f'_x(\Theta_{r,\lambda}^{t',x',h'},\alpha_r)\triangle_{h'}^iX_r^{t',x'})d \lambda dr\\
     &+\int_s^T \int_0^1 (f'_y(\Theta_{r,\lambda}^{t,x,h},\alpha_r)\triangle_h^iY_r^{t,x}-
      f'_y(\Theta_{r,\lambda}^{t',x',h'},\alpha_r)\triangle_{h'}^iY_r^{t',x'})d \lambda
      dr\\
      &+\int_s^T \int_0^1 (f'_z(\Theta_{r,\lambda}^{t,x,h},\alpha_r)\triangle_h^iZ_r^{t,x}-
      f'_z(\Theta_{r,\lambda}^{t',x',h'},\alpha_r)\triangle_{h'}^iZ_r^{t',x'})d \lambda
      dr\\
     & - \int_s^T(
     \triangle_h^iZ_r^{t,x}-\triangle_{h'}^iZ_r^{t',x'})dB_r.
       \end{split}
\end{equation*}
It is noted that
$$E ( \sup_{0 \leq s \leq T} |\triangle_h^i X_s^{t,x}-\triangle_{h'}^i X_s^{t',x'}|^p) \leq c_p
(|x-x'|^p+|h-h'|^p+|t-t'|^{\frac{p}{2}}).$$ $\forall i \in
\mathcal{M}$, using similar arguments, we can show that
  \begin{equation*}
     \begin{split}
& E \left( \sup_{0 \leq s \leq T} |\triangle_h^i
Y_s^{t,x}-\triangle_{h'}^i Y_s^{t',x'}|^p+\left(\int_{t\wedge
t'}^T|\triangle_h^i Z_s^{t,x}-\triangle_{h'}^i
Z_s^{t',x'}|^2ds\right)^{{\frac{p}{2}}}\right) \\ \leq
 & \ c_p (1+|x|^q+|x'|^q+|h|^q+|h'|^q) \times
(|x-x'|^p+|h-h'|^p+|t-t'|^{\frac{p}{2}}).
 \end{split}
\end{equation*}
The existence of a continuous derivative of $Y_s^{t,x}$ with respect
to $x$, and a mean-square derivative of $Z_s^{t,x}$ with respect to
$x$ follow from this estimation. And the existence of a continuous
second derivative of $Y_s^{t,x}$ with respect to $x$ can be proved
in a similar scheme.
 Using similar arguments as in the proof of Theorem 2.9 in
\cite{Pardoux1992}, we can show that
$\{Y_{s}^{t,x};(s,t) \in [0,T]^2, x\in R^m\}$ has an a.s. continuous
version. \qed

\bigskip

\noindent {{\bf Proof of Proposition \ref{Prop: representation of Z
of BSDE of FBSDE}:}} For any random variable $F$ of the form
$F=f(\varphi, B(h_1),\cdots,$ $B(h_n))$ with $f \in
C_0^{\infty}(R^{n})$, $\varphi \in L_{\mathcal{F}_T^{\alpha}}^2$,
$h_1,\cdots,h_n \in L^2_{\mathcal{F}_t}(0,T;R^d)$ and
$B(h_i)=\int_0^Th_i(t)dB_t$, where
$\mathcal{F}_t=\mathcal{F}^{\alpha}_{t,T} \vee \mathcal{F}_t^B$, let
$\displaystyle D_tF=\sum_{i=1}^nf'_i(B(h_1),\cdots,$ $B(h_n))h_i(t), $
$0 \leq t\leq T$. For such $F$, we define its norm as
$$\|F\|_{1,2}=\left(E\left(F^2+\int_0^T|D_tF|^2dt\right)\right)^{\frac{1}{2}}.$$
Denote $S$ as the set of random variables of the above form, we can
define sobolev space: $D^{1,2}=\bar{S}^{\|\cdot\|_{1,2}}.$ Using
the same argument in Proposition 2.3 in \cite{Pardoux1994},
we can obtain the result.\qed

%







\end{document}